\documentclass[a4paper]{amsart}
\usepackage{amssymb,latexsym}
\usepackage{amsmath,amsthm}
\usepackage{mathtools}
\usepackage{amsfonts,mathrsfs}
\usepackage{cleveref}
\usepackage{enumerate,units}
\usepackage[all]{xy}
\usepackage{graphicx}
\usepackage[normalem]{ulem}
\usepackage{url}

\theoremstyle{plain}
\newtheorem{theorem}{Theorem}[subsection]
\newtheorem{corollary}[theorem]{Corollary}
\newtheorem{lemma}[theorem]{Lemma}
\newtheorem{proposition}[theorem]{Proposition}
\newtheorem{fact}[theorem]{Fact}
\newtheorem{claim}{Claim}
\newtheorem*{claim*}{Claim}
\newtheorem*{propaux}{Proposition \theoremauxnum}
\gdef\theoremauxnum{1}

\newtheorem*{theoremaux}{Theorem \theoremauxnum}
\gdef\theoremauxnum{1}

\newenvironment{theoremff}[2][]{%
  \def\theoremauxnum{\ref{#2}}
  \begin{theoremaux}[#1]
}{%
  \end{theoremaux}
}

\newtheorem*{theorem*}{Theorem}
\newtheorem*{proposition*}{Proposition}
\newtheorem*{corollary*}{Corollary}

\theoremstyle{definition}
\newtheorem{definition}[theorem]{Definition}
\newtheorem*{definition*}{Definition}
\newtheorem{example}[theorem]{Example}
\newtheorem*{defaux}{Definition \theoremauxnum}
\gdef\theoremauxnum{1}

\newenvironment{defff}[2][]{%
  \def\theoremauxnum{\ref{#2}}
  \begin{defaux}[#1]
}{%
  \end{defaux}
}

\theoremstyle{remark}
\newtheorem*{remark}{Remark}
\newtheorem{remarkcnt}[theorem]{Remark}

\newtheorem{question}[theorem]{\textbf{Question}}

\numberwithin{equation}{section}
\newcommand{\forkindep}[1][]{%
  \mathrel{
    \mathop{
      \vcenter{
        \hbox{\oalign{\noalign{\kern-.3ex}\hfil$\vert$\hfil\cr
              \noalign{\kern-.7ex}
              $\smile$\cr\noalign{\kern-.3ex}}}
      }
    }\displaylimits_{#1}
  }
}

\DeclareRobustCommand{\claimqed}{%
  \ifmmode \mathqed
  \else
    \leavevmode\unskip\penalty9999 \hbox{}\nobreak\hfill
    \quad\hbox{\qedsymbol\text{ (claim)}}%
  \fi
}
\makeatletter
\newenvironment{claimproof}[1][\proofname]{\par
  \pushQED{\claimqed}%
  \normalfont \topsep6\p@\@plus6\p@\relax
  \trivlist
  \item[\hskip\labelsep
    #1\@addpunct{.}]\ignorespaces
}{%
  \popQED\endtrivlist\@endpefalse
}
\makeatother

\newcounter{step}                   %defines a counter that will count the steps
    {\hfill $\clubsuit$             %puts a club-symbol at the end of the step
     \vspace{7pt}\par}
\newcommand{\aff}{\text{aff}}
\newcommand{\GVar}[1][F]{(\text{SPVar}/{#1})}

\newcommand{\aRdSch}[1][F]{(\text{Rd-Sch}^{\aff}/\mathcal{O}_{#1})}
\newcommand{\aGVar}[1][F]{(\text{SPVar}^{\aff}/{#1})}
\newcommand{\aSch}[1][F]{(\text{Sch}^{\aff}/\mathcal{O}_{#1})}

\DeclareMathOperator{\Spec}{Spec}

\DeclareMathOperator{\val}{val}
\DeclareMathOperator{\res}{res}
\DeclareMathOperator{\RM}{RM}
\DeclareMathOperator{\trdeg}{tr.deg}
\DeclareMathOperator{\acl}{acl}
\DeclareMathOperator{\dcl}{dcl}

\makeatletter
\renewenvironment{proof}[1][\proofname]{\par
  \pushQED{\qed}%
  \normalfont \topsep6\p@\@plus6\p@\relax
  \trivlist
  \item[\hskip\labelsep
    #1\@addpunct{.}]\ignorespaces
  \setcounter{claim}{0}  
}{%
  \popQED\endtrivlist\@endpefalse
}
\makeatother

\begin{document}
\title{On Stably Pointed Varieties and Generically Stable Groups in ACVF}
\author{Yatir Halevi}
\address{The Fields Institute for Research in Mathematical Sciences, Toronto, Canada}
\email{ybenarih@fields.utoronto.ca}
\date{}
%\email{yatirh@gmail.com}
\thanks{The research leading to these results has received funding from the European Research Council under the European Union’s Seventh Framework Programme (FP7/2007-2013)/ERC Grant Agreement No. 291111.}

\begin{abstract}
We give a geometric description of the pair $(V,p)$, where $V$ is an affine algebraic variety over a non-trivially valued algebraically closed field $K$ with valuation ring $\mathcal{O}_K$ and $p$ is a Zariski dense generically stable type concentrated on $V$, by defining a fully faithful functor to the category of schemes over $\mathcal{O}_K$ with residual dominant morphisms over $\mathcal{O}_K$. 

We also study a maximum modulus principle on schemes over $\mathcal{O}_K$ and show that the schemes obtained by this functor enjoy it.
\end{abstract}

\maketitle

\section{Introduction}
The theory of non-trivially valued algebraically closed fields was one of the first to be studied by model theorists, going back to A. Robinson who showed it is model complete \cite{Robinson}. The theory is not stable, but over the past two decades it has been studied extensively and new methods were developed to study it and other similar theories. They fall into a wider class of theories called metastable in which there is a stable part which allows the study of certain types using stable theoretic tools (see \cite{StableDomination} and \cite{Metastable}).

The geometry in models of the theory ACF of algebraically closed fields is well understood. Each complete type concentrates on a unique irreducible algebraic variety and definable groups are definably isomorphic to algebraic groups. We still do not have a complete corresponding picture for ACVF (the theory of non-trivially valued algebraically closed fields).

In \cite[Theorem 6.11]{Metastable}, Hrushovski and Rideau-Kikuchi show that given a pair $(G,p)$, where $G$ is an affine algebraic group and $p$ is a generically stable generic type of a definable subgroup $H$, $H$ is definably isomorphic to the $\mathcal{O}$-valued points of some group scheme defined over the valuation ring $\mathcal{O}$. In this paper we expand the ideas from \cite{Metastable} to give a geometric interpretation of the category of (affine) varieties with generically stable Zariski dense types concentrated on them. We expand and elaborate:

Let $(K,\val )$ be an algebraically closed field with a non-trivial valuation. Denote by $\mathcal{O}_K=\{x\in K: \val (x)\geq 0\}$ its valuation ring, by $\Gamma_K$ its value group and by $k_K$ its residue field. Most of the following results hold for more general valued fields and we do so in the text, but for ease of presentation we state everything over $K$.

In Section \ref{s:prelimin} we give some preliminaries. We give the definition and basic properties of ACVF and generically stable types. 

In Section \ref{S:some facts} we review some basic result on schemes over valuation rings and use model theoretic tools to prove:
\begin{theoremff}{T:Surjective}
Let $\mathcal{V}$ be an irreducible scheme of finite type over $\mathcal{O}_K$. If $\mathcal{V}$ has an $\mathcal{O}_K$-point then \[\mathcal{V}(\mathcal{O}_K)\to \mathcal{V}_{k_K}(k_K)\] is surjective.
\end{theoremff}

After choosing an open affine covering, an algebraic variety over $K$ can be seen as a definable set over $K$ (see Section \ref{ss:alg.grps}). In Section \ref{S:gen stable var} we study the category $\GVar[K]$ of pairs $(V,p)$, where $V$ is an algebraic variety over $K$ and $p$ a Zariski dense generically stable $K$-definable type concentrated on it, morphisms are morphisms of varieties over $K$ that pushforward the generically stable type accordingly (Definition \ref{D:the-category_OverF}). We will call this category, the category of \emph{stably pointed varieties} (a variety together with a distinguished generically stable type). An example is $(\mathbb{A}^1_K,p_{\mathcal{O}_K})$, where $p_{\mathcal{O}_K}$ is the generic type of the closed ball $\mathcal{O}_K$.

For affine varieties $V$, we give a geometric interpretation of the pair $(V,p)$ by means of a functor $\Phi^{\aff}_K$ from $\aGVar[K]$ (the category restricted to affine varieties) to the category $\aRdSch[K]$ of affine schemes over $\mathcal{O}_K$ with residual dominant morphisms, i.e. morphisms $f:\mathcal{V}\to \mathcal{W}$ such that \[f_{k_K}:\mathcal{V}_{k_K}\to \mathcal{W}_{k_K}\] is dominant. The functor $\Phi^{\aff}_K$ is fully faithful (Proposition \ref{C:fully-faithful}), commutes with products (Proposition \ref{P:prod-two-affines}) and the objects in its image enjoy a maximum modulus principle (Propositions \ref{P:properties of Phi(V,p)}).

In Section \ref{s:gen stable grps} we consider the case of generically stable groups.

\begin{defff}{D:gen-stab-grp}
A \emph{generically stable group} is a definable group $G$ with a generically stable type $p$ concentrated on $G$ such that for any $A=\acl(A)$ over which $p$ is defined and $g\in G$, the image of $p$ under multiplication by $g$ from the left, $gp$, is definable over $A$. $G$ will be called \emph{connected} if $gp=p$ for all $g\in G$. 
\end{defff}

The typical example is $\mathrm{GL}_n(\mathcal{O}_K)$ (the invertible matrices over $\mathcal{O}_K$) as a subgroup of $\mathrm{GL}_n(K)$, and in fact we show that for any irreducible group scheme $G$ of finite type over $\mathcal{O}_K$, $G(\mathcal{O}_K)$ is generically stable (Proposition \ref{P:grp scheme over O is gen stable}). 
An irreducible algebraic group with a generically stable generic is always an Abelian variety (Proposition \ref{P:No_Ga_Gm=Abelian}).

Hrushovski and Rideau-Kikuchi give in \cite[Proposition 6.9]{Metastable} a characterization of a connected generically stable subgroup $H$ of an affine algebraic group $G$ in terms of a maximum modulus principle. The principle basically says that there exists a $K$-definable type $p$ concentrated on $H$ such that for every regular function $f$ on $G$ and $K\prec L$ a model over which $L$ is defined, there is some $\gamma_f\in\Gamma_L$ such that if $c\models p|L$ then $\val(f(c))=\gamma_f$, and for every $h\in H$, \[\val (f(h))\geq \val (f(c)).\]
In Section \ref{ss:the_functor} we generalize this notion to the non-affine case and prove the following

\begin{theoremff}{T:unique-generic-only-alg-group}
Let $\mathcal{G}$ be a separated integral group scheme of finite type over $\mathcal{O}_K$ and $p$ a $K$-definable type concentrated on $\mathcal{G}(\mathcal{O})$. The following are equivalent:
\begin{enumerate}
\item $\mathcal{G}$ has the maximum modulus principle with respect to $p$;
\item $\mathcal{G}(\mathcal{O})$ is generically stable with $p$ as its unique generically stable generic type.
\end{enumerate}
\end{theoremff}

\section{Preliminaries}\label{s:prelimin}
We will usually not distinguish between singletons and sequences thus we may write $a\in M$ and actually mean $a=(a_1,\dots,a_n)\in M^n$, unless a distinction is necessary. We use juxtaposition $ab$ for concatenation of sequences, or $AB$ for $A\cup B$ if dealing with sets. That being said, since we will be dealing with groups, we will try to differentiate between concatenation $ab$ and group multiplication $ab$ by denoting the latter by $a\cdot b$. Greek letters $\alpha,\beta, \gamma,..$ will range over the value group. Lower-case letters $a, b, c$ range over the field. 

For a field $F$, by a variety over $F$ we mean a geometrically integral separated scheme of finite type over $F$.
We will assume some basic knowledge of schemes and group schemes. They are introduced in order to be able to talk about "varieties over rings". All of this can be found in your favourite Algebraic Geometry Book, for instance \cite{gortz}.

For a definable type $p(x)$, we denote by $(d_px)\varphi(x,y)$ the $\varphi$-definition of $p$.

Finally, unless stated otherwise, $K$ will usually denote a non-trivially valued algebraically closed field.

\subsection{ACVF and Generically Stable Types}
Let $(K,\val )$ be a non-trivially valued field with value group $\Gamma_K$. The valuation ring of $K$ is the ring $\mathcal{O}_K=\{x\in K:\val (x)\geq 0\}$. It is a local ring with maximal ideal $\mathcal{M}_K=\{x\in K:\val (x)>0\}$. The residue field is the quotient $k_K=\mathcal{O}_K/\mathcal{M}_K$ and the quotient map $\res :\mathcal{O}_K\to k_K$ is called the residue map. There are different natural languages for valued fields, for now assume there is a sort for $\Gamma$ and a sort for $k$ with the obvious maps from the valued field sort.

Let ACVF be a theory stating that
\begin{enumerate}
\item $K$ is an algebraically closed field,
\item the valuation axioms,
\item the valuation is non-trivial.
\end{enumerate}

The following is well known
\begin{fact}\cite[Proposition 2.1.3]{Eli_Imag}
\begin{enumerate}
\item $k$ is a stable and stably embedded pure algebraically closed field;
\item $\Gamma$ is a stably embedded pure divisible ordered abelian group.
\end{enumerate}
\end{fact}

One may consider other languages for ACVF, such as $L_{\text{div}}$ and $L_{\Gamma}$. The language $L_{\text{div}}$ and $L_{\Gamma}$ have quantifier elimination and have the same interpretable sets. One may pass to ACVF$^{eq}$ in order to have elimination of imaginaries or add geometric sorts and get a language $L_\mathscr{G}$ admitting quantifier elimination and elimination of imaginaries (see \cite{Eli_Imag}).

We work in $L_\mathscr{G}$ in order to have elimination of imaginaries, but in practice we will ignore most of the sorts. We will mostly restrict ourselves to the following sorts: the valued field sort (which will also be called the home sort) which will be denoted by $VF$, the value group by $\Gamma$ and the residue field by $k$. 

Let $\mathcal{O}$ be the definable set defined by $\val (x)\geq 0$. 
For $M\models $ACVF, if $K=VF(M)$ then we will write $\mathcal{O}_K$ for its valuation ring and $k_K$ for its residue field. Mostly, we will treat the valued field sort and the model as interchangeable, e.g. when we say that $K$ is a model of ACVF we really mean that $K=VF(M)$ for some $M\models$ACVF. We also write
$\Gamma (A):= \dcl(A)\cap \Gamma$, and $k(A):= \dcl(A)\cap k$ and for a valued field $F$ we will write $\Gamma_F$ and $k_F$ for the value group and residue field of $F$, respectively. Notice that if $F$ is a valued field then $\Gamma(F)=\mathbb{Q}\otimes\Gamma_F$.

Let $\mathbb{U}\models$ ACVF be a monster model and $D$ a $C$-definable set. To simplify notations,  $D$ should be read as $D(\mathbb{U})$. For a $C$-definable set $D$ and $C\subseteq B$ we shall write $D(B):=D(\mathbb{U})\cap \dcl(B)$. Furthermore we denote $\mathbb{K}:=VF(\mathbb{U})$.

We continue with the definition of generically stable types.

\begin{definition}
Let $T$ be a NIP complete theory and $p$ an $A$-definable global type.
\begin{enumerate}
\item $p$ is \emph{generically stable} if \[p(x)\otimes p(y)=p(y)\otimes p(x),\]
where $p(x)\otimes p(y)=tp(a,b/\mathbb{U}))$ for $b\models p|\mathbb{U}$ and $a\models q|\mathbb{U}b$ and similarly for $p(y)\otimes p(x)$. It is also an $A$-definable type.
\item Let $Q$ be a $\emptyset$-definable set. The type $p$ is \emph{orthogonal to $Q$} if for every $A\subseteq B$ and $B$-definable function $f$ into $Q$, $f_*p$ is a constant type, where $f_*p=tp(f(a)/\mathbb{U})$ for $a\models p$.
\item Let $q$ be another $A$-definable type and $f$ an $A$-definable function. The type $p$ is dominated by $q$ via $f$ if all $A\subseteq B$, \[a\models p|B \Leftrightarrow a\models p|A \text{ and } f(a)\models q|B.\] We say $p$ is stably dominated if there exists a stable, stably embedded definable set $D$ such that $q$ is concentrated on $D$ and $p$ is dominated by $q$ via $f$.
\end{enumerate}
\end{definition}

In ACVF we have the following,
\begin{fact}\cite[Proposition 2.9.1]{Non-Arch-Tame}
Let $p$ be an $A$-definable type in ACVF, then the following are equivalent:
\begin{enumerate}
\item $p$ is stably dominated;
\item $p$ is generically stable;
\item $p$ is orthogonal to $\Gamma$.
\end{enumerate}
\end{fact}

\begin{remark}
The following are well known consequences of the definitions and some basic properties of generically stable type (see, for example, \cite[Section 2.2.2]{guidetonip}).
\begin{enumerate}
\item Products (as definable types) of generically stable types are generically stable: if $p(x)$ and $q(y)$ are generically stable then so is $p(x)\otimes q(y)$.
\item Pushforwards of generically stable types are generically stable: if $p$ is an $A$-definable generically stable type and $f$ is an $A$-definable function on $p$ then $f_*p$ is also generically stable.
\item If $p$ is dominated via $f$ by $f_*p$ and $f_*p$ is generically stable then $p$ is also generically stable.
\end{enumerate}
\end{remark}

\begin{example}\label{E:p_O-minimal-val}
Let $p_\mathcal{O}$ be the global generic type of the closed ball $\mathcal{O}$ (in the sense of \cite[Definition 7.17]{StableDomination}), i.e  $p_\mathcal{O}(x)$ says that $x\in\mathcal{O}$  but $x$ is not in any proper sub-ball of $\mathcal{O}$. Let $p_k(x)$ be the global generic type of $k$. $p_\mathcal{O}$ is stably dominated by $p_k$ via $\res$ (the residue map). One also sees that for every polynomial $f\in K[\bar X]$ and $c\models p_\mathcal{O}^n|K$, $\val (f(c))=\min_i\{\val (b_i)\}$, where $\{b_1,\dots,b_m\}$ are the coefficients of $f$.
\end{example}

\subsection{Interpreting Varieties over $K$ and Schemes over $\mathcal{O}_K$}\label{ss:interpreting}
The following is quite standard and is added mainly to fix notation and conventions.

\subsubsection{Some Words on Notation}\label{ss:definable-scheme}
We shall recall in Section \ref{ss:alg.grps} that every variety $V$ over a field $F$ gives rise to an ACF-definable set: the set of its $\mathbb{K}$-points in the monster model, where $\mathbb{K}=\mathrm{VF}(\mathbb{U})$. We shall abuse notation and also write $V$ for the definable set it defines in the monster model. For any field extension $F\subseteq L$, $V_L$ is the base-change of $V$ to $L$ and $V_L(L)$ is the set of $L$-points of $V_L$.

In Section \ref{ss:interp. Grp.Sch} we shall see that every separated finitely presented scheme $\mathcal{V}$ over $\mathcal{O}_F$ for a valued field $F$ gives rise to an ACVF-definable set: the set of $\mathcal{O}$-points in the monster model. Again we will abuse notation and write $\mathcal{V}(\mathcal{O})$ for the definable set it defines in the monster model. It is a definable subset of $\mathcal{V}_F=\mathcal{V}\times_{\mathcal{O}_F}F$. For any $(L,\mathcal{O}_L)$ valued field extension of $(F,\mathcal{O}_F)$, $\mathcal{V}_{\mathcal{O}_L}$ is the base-change of $\mathcal{V}$ to $\mathcal{O}_L$ and $\mathcal{V}_{\mathcal{O}_L}(\mathcal{O}_L)$ is the set of $\mathcal{O}_L$-points of $\mathcal{V}_{\mathcal{O}_L}$.

Similarly, in Section \ref{ss:Pro-def-grps} we will see that every quasi-compact separated scheme over $\mathcal{O}_F$ gives rise to a pro-definable set, i.e. there pro-definable set of its $\mathcal{O}$-points.

\subsubsection{Algebraic Varieties}\label{ss:alg.grps}
In the following we work in ACF, most of the following can be found in \cite[Section 7.4]{Marker}. Let $K$ be an algebraically closed field.
Let $V$ be a variety over $K$. $V$ has a finite open covering  by schemes isomorphic to $\Spec R_i$, $R_i:=K[\bar{X}]/(f_{i,1},\dots ,f_{i,s})$ where $f_{i,j}\in K[\bar{X}]$. 
Thus \[V=\bigcup_{i=1}^n V_i,\]
where $V_i=\Spec R_i$ and $(\Spec R_i)(K)\subseteq K^{n_i}$ with homeomorphisms \[\varphi_i: V_i\to \Spec R_i\] such that 
\begin{enumerate}
\item $U_{i,j}=\varphi_i(V_i\cap V_j)$ is open subscheme  of $U_i$
\item $\varphi_{i,j}=\varphi_i\circ \varphi_j^{-1}: U_{j,i}\to U_{i,j}$ is an isomorphism of varieties.
\end{enumerate}

Let $m$ be such that $(\Spec R_i)(K)\subseteq K^m$ for all $i$. Let $u_1,\dots ,u_n\in K$ be distinct and set
\[Y=\{ (x,y)\in K^{m+1}: y=u_i \text{ and } x\in (\Spec R_i)(K) \text{ for some } i\leq n\}.\] 
Define the following equivalence relation on $Y$:
\[(x,u_i)\sim (y,u_j) \Leftrightarrow u_i=u_j \wedge x=y \text{ or } u_i\neq u_j, x\in U_{i,j}\text{ and } \varphi_{i,j}(y)=x.\]

As ACF eliminates imaginaries, the quotient set $Y/\sim$ is definable. Abusing notations we will denote it by $V(K)$. Identifying $(\Spec R_i)(K)$ with $(\Spec R_i)(K)\times \{u_i\}$, the maps $\varphi_i,\varphi_{i,j}$ are definable. We will say that $W$ is an affine open subset of $V(K)$ if $W=\varphi_i^{-1}(X)$ for some atlas $\{\varphi_i, \Spec R_i \}$ of $V$ and open affine $X\subseteq (\Spec R_i)(K)$. An open subset of $V(K)$ is a finite union of affine open subsets, so also definable.

\begin{remarkcnt}\label{R:cover-immat}
Note that the choice of the affine open cover is immaterial. Since the category of varieties with a distinguished open affine cover is equivalent to the category of varieties, the above construction gives a unique definable set upto definable isomorphisms.
\end{remarkcnt}

Assume that $G=V$ is an algebraic group. In this case $G(K)$ is a definable group, indeed the morphism \[m: G(K)\times G(K)\to G(K)\] is continuous, thus for every $i$, $m^{-1}(V_i)$ is a finite union of open affines and $m$ is determined by its restriction to these open affines. Each of these restrictions is definable. Similarly for inversion, so $G(K)$ is an definable group. $G(K)$ has generic types.

\begin{remark}
\begin{enumerate}
\item $G(K)$ is connected if and only if it is irreducible.
\item Since the map $Y\to G(K)$ is a definable finite-to-one surjection, \[\RM (Y)=\RM (G(K)).\] If $G(K)$ is connected then it has a unique generic type and 
\[\dim G=\RM (G(K))=\trdeg (K(a)/K),\]
for $a\models p|_K$, where $p$ is the generic type of $G(K)$.
\end{enumerate}
\end{remark}

From now on we also denote by $V$ the set the variety defines in the monster model, i.e. $V(\mathbb{U})$. Similarly for algebraic groups. 

%By genericity, there exists a $K$-definable affine open subset $V$ of $G$ and $a_1, \dots, a_n\in G(K)$ such that 
%\[G=\bigcup_{i=1}^n a_i\cdot V.\]
%Let $\varphi: V\to (SpecK[F])(K)$ be the atlas map homeomorphism and $\pi$ its inverse (the quotient map from before), they are both definable. We may assume that $V=V^{-1}$ and that $e\in V$. 
%
%We would like to understand the variety structure on each of the $a_i\cdot V.$ Multiplication by $a_i$ is an isomorphism, so each of the $a_i\cdot V$ are affine with a $K$-algebra of regular functions of the form $a_i\cdot f:=f(a_i^{-1}x)$ for $f\in K[F]$ (of course, formally we mean $f(\varphi(a_i^{-1}x))$, but we will leave it this way for now). Denote this $K$-algebra by $a_i\cdot K[F]$ and let \[\varphi_i: a_i\cdot V \to (Spec (a_i\cdot K[F]))(K)\] to be the appropriate homeomorphism.
%
%If $\left\{ Spec K[F]_{(h_j)}: h_j\in K[F]\right\}$ is an open covering by basic open sets of $Spec K[F]$ then \[\left\{ (a_i\cdot Spec K[F])_{(a_i\cdot h_j)}: h_j\in K[F]\right\}\] is a covering of $Spec (a_i\cdot K[F])$.
%
%Notice that we have a in isomorphism \[SpecK[F]_{(h)}\xrightarrow{\cong} Spec(a\cdot K[F])_{(a\cdot h)},\] given by the comoprhisms:
%\[\frac{f(a^{-1}x)}{h^n(a^{-1}x)}\xrightarrow{x\mapsto ax}\frac{f(x)}{h^n(x)}.\]

\begin{remarkcnt}\label{R:ease of notations}
Technically, in order to talk about regular functions on affine open subsets of $V$ one must use the atlas maps. If $U\subseteq V$ is an affine open subset then there exits an atlas map $(\varphi,\Spec R)$, and some affine open subset $X\subseteq (\Spec R)(K)$ with $\varphi^{-1}(X)=U(K)$, where $(\varphi,R)$ is one of the $(\varphi_i,R_i)$ from above. A regular function on $U$ is of the sort $f(\varphi(x))$, where $f$ is a regular function on $X$.  We will omit the $\varphi$ and write $f(x)$ for ease of writing. As a consequence when writing for $f$ a regular function on $U$ and $p$ a definable type on $U$ \[(d_px)(\val (f(x))\geq \gamma)\] we actually mean \[(d_{\varphi_*p}x)(\val (f(x))\geq \gamma).\]

Recall that the pushforward of a generically stable, by a definable function, is also generically stable. As a result, since $\varphi$ is a definable isomorphism, $p$ is generically stable if and only if $\varphi_*p$ is generically stable.
\end{remarkcnt}

\subsubsection{Schemes over $\mathcal{O}_K$}\label{ss:interp. Grp.Sch}

Let $K$ be an algebraically closed valued field with valuation ring $\mathcal{O}_K$. 
Let $\mathcal{V}$ be a separated finitely presented scheme over $\Spec \mathcal{O}_K$. $\mathcal{V}$ has a finite open covering  by schemes isomorphic to $\Spec R_i$, $R_i:=\mathcal{O}_K[\bar{X}]/(f_{i,1},\dots f_{i,n})$ where $f_{i,j}\in \mathcal{O}_K[\bar{X}]$. 
$\mathcal{V}_K:=\mathcal{V}\times_{\mathcal{O}_K}K$ is thus a scheme of finite type over $K$ and we may identify $\mathcal{V}_K(K)$ with a definable set (in ACF).

The $\mathcal{O}_K$-points of a finitely presented affine scheme over $\mathcal{O}_K$, $\Spec \mathcal{O}_K[\bar{X}]/(f_1,\dots,f_n)$, may be identified with the definable set (in ACVF): \[\{\bar{a}\in\mathcal{O}^{|\bar{X}|}_K: f_1(\bar{a})=\dotso=f_n(\bar{a})=0\}.\]
In general there is a map $\mathcal{V}(\mathcal{O}_K)\to \mathcal{V}_K(K)$, but, since $\mathcal{V}$ is separated, by the valuative criterion of separatedness \cite[Theorem 15.8]{gortz}:

\begin{fact}
$\mathcal{V}(\mathcal{O}_K)\hookrightarrow \mathcal{V}_K(K)$.
\end{fact}

Since a $K$-point of $\mathcal{V}_K$ is an $\mathcal{O}_K$-point of $\mathcal{V}$ if and only if it is an $\mathcal{O}_K$-point of $\Spec R_i$ for some $i$, $\mathcal{V}(\mathcal{O}_K)$ is a definable subset of $\mathcal{V}_K(K)$.

\begin{remark}
As in Remark \ref{R:cover-immat}, the choice of the open affine cover is immaterial.
\end{remark}

Similarly to Section \ref{ss:alg.grps}, if $\mathcal{V}=\mathcal{G}$ is a separated finitely presented group scheme over $\mathcal{O}_K$, $\mathcal{G}(\mathcal{O}_K)$ is a definable subgroup of $\mathcal{G}_K(K)$.

\subsubsection{Schemes over $\mathcal{O}_K$ as Pro-Definable Sets}\label{ss:Pro-def-grps}
Not all schemes are finitely presented, but that does not mean we can not access them within the definable world.

\begin{fact}\cite[Tag 01YT]{stacks-project}\label{F:directed-limit}
Let $S$ be a quasi-compact quasi-separated scheme over an affine scheme $R$, i.e. $S$ can be covered by a finite number of affine open subschemes, any two of which have intersection also covered by a finite number of affine open subschemes.

\begin{enumerate}
\item There exists a directed inverse system of schemes $(S_i, \pi_{ij})$ with affine transition maps such that each $S_i$ is a finitely presented schemes over $R$ and $S=\varprojlim_i S_i$.
\item For every quasi-compact open $U\subseteq S$ there exists $i_U$ and opens $U_i\subseteq S_i$ such that $\pi_i^{-1}(U_i)=U$ and $U=\varprojlim_{i\geq i_U} U_i$, where $\pi_i:S\to S_i$ is the natural projection morphism.

\end{enumerate}
\end{fact}

Let $K\models$ ACVF and $\mathcal{V}$ be a quasi-compact separated scheme over $\mathcal{O}_K$. By the above fact there exist finitely presented schemes over $\mathcal{O}_K$, $\mathcal{V}_i$, such that $\mathcal{V}=\varprojlim_i \mathcal{V}_i$. Since $\mathcal{V}$ is separated, there exists $i_0$ such that for all $i\geq i_0$, $\mathcal{V}_i$ is separated (\cite[Tag 086Q]{stacks-project}). Hence we may assume all the $\mathcal{V}_i$ are separated.

Observe that $\mathcal{V}_K:=\mathcal{V}\times_{\mathcal{O}_K}K=\varprojlim_i (\mathcal{V}_i)_K$. Following Section \ref{ss:interp. Grp.Sch} we may identify each of the $\mathcal{V}_i(\mathcal{O})$ with a definable subset of the corresponding $(\mathcal{V}_i)_K$. The transition maps correspond to definable maps. Thus we may identify $\mathcal{V}_K$ with a pro-definable set (in ACF) and $\mathcal{V}(\mathcal{O})$ with a pro-definable subset (in ACVF) of $\mathcal{V}_K$. 

Notice that a type of an element of $\mathcal{V}$, $p$, is a compatible (with respect to the transition maps) sequence of types $(p_i)_i$ such that $p_i$ is concentrated on $\mathcal{V}_i$.

If $U\subseteq \mathcal{V}_K$ is an open subscheme with $U=\varprojlim_{i\geq i_U} U_i$ then a regular function $f$ on $U$ corresponds to a regular function on some $U_i$. The following is easy.

\begin{lemma}
\begin{enumerate}
\item Such a type $p=(p_i)_i$ is Zariski dense in $\mathcal{V}_K$ if and only if $p_i$ is Zariski dense in $(\mathcal{V}_i)_K$ for all $i$.
\item If $\mathcal{V}$ is irreducible (respectively, integral) then there exits $i_0$ such that for all $i\geq i_0$, $\mathcal{V}_i$ is irreducible (respectively, integral).

\end{enumerate}
\end{lemma}

A definable a function $f$ on $\mathcal{V}$ is a compatible (with respect to the transition maps) sequence of definable functions $(f_i)_i$ such that $f_i$ is a definable function on $\mathcal{V}_i$. Since the system $(\mathcal{V}_i,\pi_{ij})$ is directed, a definable function on $p$ is a definable function on one of the $p_i$, and consequently:
\begin{lemma}
A type $p=(p_i)_i$ on $\mathcal{V}$ is stably dominated if and only if $p_i$ is stably dominated for each $i$.
\end{lemma}
%\begin{proof}
%If $p$ is orthogonal to $\Gamma$ then obviously $p_i$ is orthogonal to $\Gamma$. For the other direction, let $f(x_{i_1},\dots x_{i_k})$ be a definable function on $p$ with $i_1,\dots, i_k\in I$, $c=(c_i)\models p$ (so $c_i\models p_i$) and $i_1,\dots,i_k\leq j\in I$. Thus if we denote by $\pi_{i,j}:S_j\to S_i$ the transition morphisms, $f(\pi_{i_1,j}(x_{j}),\dots,\pi_{i_k,j}(x_{j}))$ is a definable function on $p_j$ so there exists $\gamma$ such that 
%\[\models (d_{p_j}x_j)\val (f(\pi_{i_1,j}(x_{j}),\dots,\pi_{i_k,j}(x_{j})))=\gamma.\] 
%Consequently \[\val (f(c_{i_1},\dots,c_{i_k}))=\val (f(\pi_{i_1,j}(c_{j}),\dots,\pi_{i_k,j}(c_{j})))=\gamma.\]
%\end{proof}

Similarly, if $\mathcal{G}$ is a quasi-compact separated group scheme over $\mathcal{O}_K$ then one may interpret $\mathcal{G}_K$ as a pro-definable group with $\mathcal{G}(\mathcal{O})$ as a pro-definable subgroup. 

\section{Some Facts on Schemes over Valuation Rings}\label{S:some facts}

\subsection{Facts on Schemes over General Valuation Rings}\label{ss: general facts on..}
Let $R$ be a valuation ring. The following properties are well known.

\begin{fact}\cite{nagata1966}\label{F:tors-free}
$M$ is a flat $R$-module if and only if it is torsion-free.
\end{fact}

\begin{fact}\cite{nagata1966}\label{F:f.t-f.p}
A flat scheme of finite type over a valuation ring is finitely presented.
\end{fact}

The following is straightforward.
%\begin{lemma}
%Let $V$ be a scheme over $SpecR$ and $s:SpecR\to V$ a $R$-point. If $s(\mathbf{m})\subseteq U$ where $\mathbf{m}$ is the special point of $SpecR$ and $U$ is an open subscheme of $V$ then $s(SpecR)\subseteq U$.
%\end{lemma}

\begin{fact}\label{L:dom+red=injective}
Let $X=\Spec A$ and $Y=\Spec B$ be affine schemes with $Y$ reduced. Then $X\to Y$ is dominant if and only if $B\to A$ is injective.
\end{fact}

Recall that if $\mathcal{V}$ is a scheme over $R$, then a $R$-point corresponds to a section $s:\Spec R\to \mathcal{V}$ of $\mathcal{V}\to \Spec  R$.

\begin{proposition}\label{P:faith-flat}
Let $\mathcal{V}$ be an integral scheme over $R$. If $\mathcal{V}$ has a $R$-point then $\mathcal{V}$ is faithfully flat over $R$.
\end{proposition}
\begin{proof}
Let $\mathcal{U}=\Spec A$ be an affine open subscheme of $\mathcal{V}$ containing the given $R$-point. Since the $R$-point gives a section of $\mathcal{U}\to \Spec R$, the corresponding homomorphism $R\to A$ is injective. The following Claim proves that $\mathcal{U}$ is flat over $R$.
\begin{claim}
Let $A$ be an integral domain and \[R\xrightarrow{\rho} A\] a ring homomorphism. If $\rho$ is injective then $A$ is a $R$-flat module via $\rho$.
\end{claim}
\begin{claimproof}
Assume $\rho$ is injective. Since $R$ is a valuation ring, if $A$ is not a flat $R$ module, by \Cref{F:tors-free}, there exist $a$ and $r$, both non-zero, such that $\rho(r)a=0$. Since $A$ is an integral domain, $\rho(r)=0$, so $r=0$, contradiction.
\end{claimproof}

%Notice that, in the notations of the Claim, $\rho^{-1}(\mathfrak{p})=(0)$ where $\mathfrak{p}$ is the minimal prime ideal of $A$. 
Also, since $R$ is reduced, by the previous lemma, $\mathcal{U}\to \Spec R$ is dominant and hence, since $\mathcal{V}$ is integral, so is $\mathcal{V}\to \Spec R$. 

Because flatness is open on the source it is enough to show that $X$ is flat over $R$ for every affine open subscheme $X$ of $\mathcal{V}$. For such an $X$, since $\mathcal{V}$ is integral, $X\to \Spec R$ is also dominant and hence flat by Lemma \ref{L:dom+red=injective} and the Claim.
% Becasue a minimal prime ideal consists of zero-divisors. 
%Now, let $V$ be an irreducible scheme over $R$. 
%By the above, generic point of $V$ gets sent to the generic point of $R$. 

 % If $SpecB$ is any affine open subscheme of $V$, the morphism $R\to B$ must be injective (since $R$ is reduced) and we may apply the Claim again. 
$\mathcal{V}$ is faithfully flat because $\mathcal{V} \to \Spec R$ is surjective by the existence of the section.
\end{proof}

%As a consequence of the proof we also get the following:
%Nope - assuming here separatedness, since using the fact that the section comes from a K-point
%\begin{proposition}\label{P:separated}
%Let $G$ be an irreducible group scheme of finite type over $SpecR$. Then $G$ is separated over $R$.
%\end{proposition}
%\begin{proof}
%By considering the following:
%\[\xymatrix{
%G \ar[r]_-{\Delta_{G/SpecR}} \ar[d] &
%G \times_{SpecR} G \ar[d]^{(g, g') \mapsto g^{-1}g'} \\
%SpecR \ar[r]^e & G
%}\]
%we see that it is enough to show that $e$ is a closed immersion (in fact it is if and only if). $G$ has a $R$-point so by \Cref{P:faith-flat} and \Cref{T:f.t-f.p} it is finitely presented. Since closed immersions are open on the target it is enough to show that if $R[\bar{X}]/I$ is one of the open affine coverings of $G$ then 
%\[R[\bar{X}]/I\xrightarrow{f\mapsto f(e)}R\]
%is surjective ($SpecA\to SpecB$ is a closed immersion if and only if $B\to A$ is a surjection), but this is straightforward since \[R\hookrightarrow R[\bar{X}]/I.\]
%\end{proof}

Since we will mostly deal with integral schemes over $\Spec R$ that do have $R$-points, we will usually use "of finite type" as our finiteness condition remembering that it implies "finitely presented".

\begin{proposition}\label{P:V irr then general}
Let $F$ be a valued field with valuation ring $R$. Let $\mathcal{V}$ be an irreducible scheme over $R$, with $\mathcal{V}\to \Spec R$ dominant. Then 
\[\mathcal{V}\times_{\Spec R} \Spec F\] is irreducible as well.
\end{proposition}
\begin{remark}
By the proof Proposition \ref{P:faith-flat}, if $\mathcal{V}$ is integral and has a $R$-point then $\mathcal{V}\to \Spec R$ is dominant.
\end{remark}
\begin{proof}
Since the generic fiber $\mathcal{V}\times_{\Spec R}\Spec F$  is dense in $\mathcal{V}$ it is also irreducible.
\end{proof}

\subsection{Facts on Schemes over $\mathcal{O}_K$}
Let $\mathcal{V}$ be a scheme over a valuation ring $R$ with residue field $k$. Let  $s:\Spec R\to \mathcal{V}$ be an $R$-point, i.e. a section of $\mathcal{V}\to \Spec  R$. Base changing with $\Spec k\to \Spec R$, gives a $k$-point $\Spec k\to \mathcal{V}_k$ of $\mathcal{V}_k$. This defines a map \[r:\mathcal{V}(R)\to \mathcal{V}_k(k).\]

We move to algebraically closed valued fields. Let $K\models$ ACVF and $\mathcal{O}_K$ its valuation ring.

In this section we will prove that if $\mathcal{V}$ is a quasi-compact separated irreducible scheme over $\mathcal{O}_K$, with an $\mathcal{O}_K$-point, then the map $r:\mathcal{V}(\mathcal{O})\to \mathcal{V}_k$ is surjective. We will prove it using arguments similar to the ones given in \cite[Lemma 6.6]{Metastable}.

\begin{remark}
By Fact \ref{F:tors-free}, an affine scheme $\Spec \mathcal{O}_K[\bar{X}]/I$ of finite type over $\mathcal{O}_K$ is flat over $\mathcal{O}_K$ if and only if $I=IK[\bar{X}]\cap \mathcal{O}_K[\bar{X}]$.
\end{remark}

The following is well known for varieties over $K$, but the same proof gives
\begin{fact}
Let $V=\Spec A$ be an irreducible affine scheme over $K$. It determines a pro-definable set in ACF, and has a unique generic type (a type which is not concentrated on any closed subvariety) $p$. Furthermore, if $b\models p|K$ then $A/Nil(A)\cong K[b]$, such an element $b$ is called a $K$-generic of $V$.
\end{fact}
%\begin{proof}
%Let $R_i\subseteq R$ be finitely generated sub $K$-algebras such that $R=\bigcup_i R_i$. Let \[V=\varprojlim_i V_i=\varprojlim_i SpecK[\bar{X}_i]/I_i=K[\bar{X}]/\bigcup_i I_i\] be the corresponding inverse limit. Denote by $f_{ij}$ the dominant transition maps.  Let $P(\bar{x})$ be partial type saying that $f_{ij}(\bar{x}_j)=\bar{X}_j$ and that $\bar{x}_i$ is the generic type of $V_i$. Since the index set is directed and the transition maps are dominant, $P(\bar{x})$ is finitely  satisfiable. Let $p(\bar{x})$ be the corresponding generic type. Let $b=(b_i)_i\models p$. Mapping $\bar{X}_i\mapsto b_i$ gives the desired isomorphism.
%\end{proof}

%In the following I deleted $V_K$ is irreduicble
\begin{lemma}\label{L:Iso_to_generic}
Let $\mathcal{V}=\Spec A$ be an irreducible affine scheme which is flat over $\mathcal{O}_K$ and let $b$ be a $K$-generic point of $\mathcal{V}_K$. If $\mathcal{V}\to \Spec R$ is dominant then
\begin{enumerate}
\item $A/Nil(A)\cong \mathcal{O}_K[b];$
\item and if we further assume that $\mathcal{V}$ is integral and $\mathcal{V}_{k_K}$ has a $k_K$-point then for any $K\subseteq L$, where $L$ is an algebraically closed field for which $b\in L$, there exists a valuation ring $S\subseteq L$ such that $S\cap K=\mathcal{O}_K$ and $b\in S$. 
\end{enumerate}
\end{lemma}
\begin{proof}
\begin{enumerate} 
\item Let $A\cong \mathcal{O}_K[\bar{X}]/I$ and let $\mathcal{O}_K[\bar{X}]\to \mathcal{O}_K[b]$ be the natural surjection. If $f\in\mathcal{O}_K[\bar{X}]$ is such that $f(b)=0$ then since $\mathcal{V}_K$ is irreducible (see \Cref{P:V irr then general}) and $b$ is a $K$-generic point, $f^n\in IK[\bar{X}]$ and by flatness, $f^n\in I$, so $f\in \sqrt{I}$.
\item  Every point in $\mathcal{V}_{k_K}(k_K)$ arises from a $k_K$-algebra homomorphism \[\mathcal{O}_K[\bar X]/I\otimes_{\mathcal{O}_K}k_K\to k_K. \] By $(1)$ this induces a $k_K$-algebraic homomorphism
\[h:\mathcal{O}_K[b]\otimes_{\mathcal{O}_K}k_K\to k_K. \]
\begin{claim}
$\mathcal{M}_K=\{x\in K:\val (x)>0\}$ generates a proper ideal of $\mathcal{O}_K[b]$.
\end{claim}
%You need the k-point because otherwise we could have that in O[x]/I that 1=m for some m in the maximal ideal, e.g. I=(1+m).
\begin{claimproof}
Otherwise, for some finite number of $m_i\in\mathcal{M}_K$ and $f_i\in\mathcal{O}_K[\bar{X}]$ we would have \[\sum_i m_if_i(b_i)=1,\] but then applying $\otimes_{\mathcal{O}_K}k_K$ and then applying $h$ we reach a contradiction.
\end{claimproof}
Thus we may extend $\mathcal{M}_K$ to a maximal ideal $M$ of $\mathcal{O}_K[b]$. By \cite[Theorem 10.2]{matsumura}, there exists a valuation ring $\mathcal{O}_K[b]\subseteq \mathcal{O}_L\subseteq L$ such that $\mathcal{M}_L\cap \mathcal{O}_K[b]=M$, where $\mathcal{M}_L$ is the maximal ideal of $\mathcal{O}_L$, and in particular $\mathcal{M}_L\cap \mathcal{O}_K=\mathcal{M}_K$. 
%We want $\mathcal{M}_L\cap\mathcal{O}_K=\mathcal{M}_K$
As a consequence, the valuation on $K$ can be extended to $L$ in such a way that $b\in \mathcal{O}_L$.
\end{enumerate}
\end{proof}

%Erased $V_K$ is irreducible
The following was proved in \cite[Lemma 6.6]{Metastable} for the affine of finite type case.
\begin{proposition}\label{P:O-points-are-dense}
Let $\mathcal{V}$ be an irreducible quasi-compact separated scheme over $\mathcal{O}_K$. If $\mathcal{V}$ has an $\mathcal{O}_K$-point then $\mathcal{V}(\mathcal{O})$ is Zariski dense in $\mathcal{V}_K$.

Moreover, if $\mathcal{V}$ is of finite type over $\mathcal{O}_K$ then $\mathcal{V}(\mathcal{O}_K)$ is Zariski dense in $\mathcal{V}_K(K)$.
\end{proposition}
\begin{proof}
Since every open subscheme of $\mathcal{V}_K$ is Zariski dense and the given $\mathcal{O}_K$-point must land in some affine open subscheme of $\mathcal{V}$, we may reduce to the case where $\mathcal{V}$ is affine. 

Assume that $\mathcal{V}$ is affine. We will show that any basic open subscheme of $\mathcal{V}_K$ has an $\mathcal{O}$-point. Let $f$ be a regular function on $\mathcal{V}_K$ and assume it is defined over $K^\prime$ for some $K\prec K^\prime$. Since $\mathcal{V}$ has an $\mathcal{O}_K$ point so does the reduced induced subscheme $\mathcal{V}_{red}$ and consequently $(\mathcal{V}_{k_K})_{red}$ has a $k_K$-point. By Lemma \ref{L:Iso_to_generic} there exists $K^\prime\prec L$ and $b\in \mathcal{O}_L$ which is a $K^\prime$-generic of $\mathcal{V}_{K'}(K')$. Since $\mathcal{V}_K$ is irreducible (again, by \Cref{P:V irr then general}) and thus has a unique generic type, $f(b)\neq 0$ as required.  If $\mathcal{V}$ is of finite type over $\mathcal{O}_K$, there thus exists $b^\prime\in \mathcal{V}(\mathcal{O}_{K'})$ with $f(b)\neq 0$.
\end{proof}
%
%\begin{corollary}\label{C:O-points-are-dense}
%Let $\mathcal{V}$ be a separated scheme of finite type over $\mathcal{O}_K$. If $\mathcal{V}$ is irreducible and has an $\mathcal{O}_K$-point then $\mathcal{V}(\mathcal{O}_K)$ is Zariski dense in $\mathcal{V}_K(K)$.
%\end{corollary}
%\begin{proof}
%
%\end{proof}
%
%\begin{corollary}\label{C:O-points dense in qcqs}
%Let $\mathcal{V}$ be an irreducible quasi-compact separated scheme over $\mathcal{O}_K$. If $\mathcal{V}$ has an $\mathcal{O}_K$-point then $\mathcal{V}(\mathcal{O})$ is Zariski dense in $\mathcal{V}_K$.
%\end{corollary}
%
%\begin{proof}
%As in Corollary \ref{C:O-points-are-dense}, we may assume that $\mathcal{V}$ is affine.
%
%By Section \ref{ss:Pro-def-grps} we may write $\mathcal{V}=\varprojlim_i \mathcal{V}_i$ with $\mathcal{V}_i$ affine, irreducible and of finite type over $\mathcal{O}_K$. Let $U\subseteq \mathcal{V}_K$ be an open subscheme, thus we may write $U=\varprojlim_i U_i$ where $U_i\subseteq (\mathcal{V}_i)_K$ (Fact \ref{F:directed-limit}). Since each $U_i$ has an $\mathcal{O}_K$-point of $\mathcal{V}_i$, by compactness, $\mathcal{V}(\mathcal{O})\cap U\neq \emptyset$.
%\end{proof}

%Erased $V_K$ is irreducible
\begin{theorem}\label{T:Surjective}
Let $\mathcal{V}$ be a quasi-compact separated irreducible scheme over $\mathcal{O}_K$. If $\mathcal{V}$ has an $\mathcal{O}_K$-point then \[r:\mathcal{V}(\mathcal{O})\to \mathcal{V}_{k}\] is surjective.

Moreover, if $\mathcal{V}$ is of finite type over $\mathcal{O}_K$ then 
\[r:\mathcal{V}(\mathcal{O}_K)\to \mathcal{V}_{k_K}(k_K)\] is surjective.
\end{theorem}
\begin{proof}
Since $\mathcal{V}$ has an $\mathcal{O}_K$ point, $\mathcal{V}_{k_K}$ has a $k_K$-point. Every $\mathcal{O}$-point (resp. $k$-point) factors through $\mathcal{V}_{red}$ (resp. $(\mathcal{V}_{k})_{red}$) so we may assume that $\mathcal{V}$ and $\mathcal{V}_{k}$ are reduced. By Proposition \ref{P:faith-flat}, $\mathcal{V}\to \mathcal{O}_K$ is faithfully flat. We first assume that $\mathcal{V}$ is affine, say $\mathcal{V}=\Spec \mathcal{O}_K[\bar{X}]/I$, where $|\bar X|$ may be infinite but small.

Let $a\in \mathcal{V}_{k}$ be a $k$-point of $\mathcal{V}_{k}$. After base-changing, for simplicity, we may assume that $a\in \mathcal{V}_{k_K}(k_K)$. Let $b\in L$ be a $K$-generic of $\mathcal{V}_K(K)$ in some elementary extension $K\prec L$. By Lemma \ref{L:Iso_to_generic} we may assume that $b\in \mathcal{O}_L$.
The map \[\mathcal{O}_K[\bar{X}]/I\to \mathcal{O}_K[\bar{X}]/I\otimes_{\mathcal{O}_K} k_K\to k_K\] sending $f\mapsto \bar{f}\mapsto \bar{f}(a)$ extends the residue map $\res :\mathcal{O}_K\to k_K$. 
Since $\mathcal{O}_K[b]\cong \mathcal{O}_K[\bar{X}]/I$ (Lemma \ref{L:Iso_to_generic}), in some elementary extension an $\mathcal{O}$-point of $\mathcal{V}$ maps to $a$. If $\mathcal{V}$ is of finite type over $\mathcal{O}_K$, we can find $b^\prime\in \mathcal{V}(\mathcal{O}_K)$ such that $\res (b^\prime)=a$.

If $\mathcal{V}$ is not affine, the $k_K$-point of $\mathcal{V}_{k_K}$ must lie in some affine open \[\mathcal{U}_k:=\mathcal{U}\times_{\Spec \mathcal{O}_K}\Spec k\] for some $\mathcal{U}\subseteq \mathcal{V}$ affine open subscheme. Since $\mathcal{U}$ is also irreducible and flat over $\mathcal{O}_K$, by Lemma \ref{L:Iso_to_generic}, we may reduce to the affine case.
\end{proof}
%
%\begin{corollary}\label{C:surjectivity for not-of-finite-type}
%Let $\mathcal{V}$ be a quasi-compact separated irreducible scheme over $\mathcal{O}_K$. If $\mathcal{V}$ has an $\mathcal{O}_K$-point then $\mathcal{V}(\mathcal{O})\to \mathcal{V}_k(k)$ is surjective.
%\end{corollary}
%\begin{proof}
%By Section \ref{ss:Pro-def-grps} we may write $\mathcal{V}=\varprojlim_i \mathcal{V}_i$ with $\mathcal{V}_i$ separated irreducible and of finite type over $\mathcal{O}_F$.  
%By Theorem \ref{T:Surjective}, and compactness, there exists $b\in \mathcal{V}(\mathcal{O})$ such that $r(b)=a$.
%
%
%%Let $a=(a_i)_i\in  V_k(k)$ and let $K$ be small model containing $a$. 
%
%%Consider the partial type $P(\bar{x})$ saying that $f_{ij}(x_j)=x_i$ and that $r_i(x_i)=a_i$, where $r_i=V_i(\mathcal{O})\to (V_i)_k(k)$ and the $f_{ij}$ are the transition maps. 
%\end{proof}

\section{Stably Pointed Varieties}\label{S:gen stable var}
Let $F$ be a perfect Henselian valued field with valuation ring $\mathcal{O}_F$. In ACVF, it means that $F$ is definably closed ($\dcl_{VF}(F)=F$). Some of what follows may probably be done in a higher level of generality. Let $K\models$ACVF, usually a model containing $F$.

Let $V$ be a variety over $F$. As was said before, and hopefully without creating too much confusion, we will also denote by $V$ the definable set $V$ defines in the monster model.

In ACF the pair $(V,p)$ where $V$ is an affine variety over $K$ and $p$ is a Zariski dense type concentrated on $V$ is well understood. There is a unique such $p$, and, by assumption, it is the unique generic type of $V$. We would like to develop an analogous picture for ACVF. 

We first define a functor $\Phi_F$ from the pairs $(V,p)$ of an affine variety and a Zariski dense generically stable type over a definably closed field $F$ to schemes over $\mathcal{O}_F$,  \[(V,p)\mapsto \Phi_F(V,p).\] In general there is no connection between $\Phi_F$ and $\Phi_L$ for some valued field extension $F\subseteq L$.

In Proposition \ref{P:desc-between-models} we will show that if the pair $(V,p)$ is taken over a model, the choice of the model is immaterial, namely:
\[\Phi_K(V,p)\times_{\mathcal{O}_K}\mathcal{O}_L=\Phi_L(V_L,p)\] for any $K\prec L$ models of ACVF and thus they give the same definable set $\Phi_K(V,p)(\mathcal{O})$.
Combining with Corollary \ref{C:descent-for-alg} we get that the same happens for pairs $(V,p)$ over a definably closed defectless non-trivially valued field $F$ with perfect residue field for which $\Gamma_{F(c)}=\Gamma_F$ for $c\models p|F$, i.e. \[\Phi_F(V,p)\times_{\mathcal{O}_F}\mathcal{O}_L=\Phi_L(V_L,p)\] for any model $F\subseteq L$. 

In view of this discussion, if the base can be understood from the context or more importantly if the base is nice enough and we only care about model theoretic properties, e.g. properties of the definable sets of $\mathcal{O}$-points or $\mathbb{K}$-points, we may drop the subscript and denote this functor by $\Phi$, without specifying the base.

Finally  we study the functor for $(V,p)$ with $p$ strongly stably dominated (see \ref{ss:finiteness conditions}).

\subsection{Stably Pointed Varieties Exist}
The main objects of this section will be pairs $(V,p)$, where $V$ is an algebraic variety and $p$ is a Zariski dense generically stable type concentrated on it. The aim of this subsection is to prove that these pairs can be found in abundance.

\begin{example}
Let $K\models$ ACVF and let $p_{\mathcal{O}}$ be the generic type of the closed ball (see Example \ref{E:p_O-minimal-val}). It is Zariski dense in $\mathbb{A}^1_K$. Notice that although $\mathbb{A}^1_K$ is integral it does not have a unique Zariski dense generically stable type, for instance $ap_{\mathcal{O}}$ (the push-forward of $p_{\mathcal{O}}$ with respect to multiplication by $a$) for $a\in K^\times$ with $\val (a)>0$ is Zariski dense, generically stable and not equal to $p_{\mathcal{O}}$.
\end{example}

\begin{proposition}
For any valued field $F$ and variety $V$ over $F$ there exists a Zariski dense generically stable type concentrated on $V$.
\end{proposition}
\begin{proof}
By the Noether normalization lemma there exists a surjective finite dominant morphism $f:V\to \mathbb{A}^d_F$ for some $d$. Since $V$ is geometrically irreducible and $f$ is dominant, there exists a Zariski dense global type on $p$ on $V$ such that $f_*p=p_\mathcal{O}^d$. It is $\acl(F)$-definable since $f$ is quasi-finite, see \cite[Lemma 2.3.4]{Non-Arch-Tame}. The type $p$ is generically stable (in fact, strongly stably dominated) by \cite[Proposition 8.1.2]{Non-Arch-Tame}.
\end{proof}

Even more can be said if we start with a scheme of finite type over $\mathcal{O}_F$. We will need to recall the following:

\begin{fact}\cite[Tag 0B2J]{stacks-project}\label{F:dim-generic=special}
Let $f:X\to \Spec R$ be a morphism from an irreducible scheme to the spectrum of a valuation ring. If f is locally of finite type and surjective, then the special fibre is equidimensional of dimension equal to the dimension of the generic fibre.
\end{fact}

\begin{proposition}\cite[Essentially in Proposition 8.1.2]{Non-Arch-Tame}\label{P:actually stab.dom}
Let $\mathcal{V}$ be an integral separated scheme over $\mathcal{O}_F$ and $p$ a generically stable type concentrated on $\mathcal{V}(\mathcal{O})$ definable over an algebraically closed valued field $F\subseteq K$. If $\dim \mathcal{V}_K=\dim\mathcal{V}_k$ and $r_*p$ is a generic type of $\mathcal{V}_k$, then  $p$ is stably dominated by $r_*p$ via $r:\mathcal{V}(\mathcal{O})\to \mathcal{V}_k$.

In particular, the result holds if we replace the assumption $\dim \mathcal{V}_K=\dim\mathcal{V}_k$  with $\mathcal{V}$ being of finite type over $\mathcal{O}_F$.
\end{proposition}
\begin{remark}
Since $\mathcal{V}(\mathcal{O})\neq \emptyset$, if $\mathcal{V}_{\mathcal{O}_L}$ is of finite type over $\mathcal{O}_L$ then it is flat over $\mathcal{O}_L$, for any algebraically closed valued field $L$ containing $F$, by Proposition \ref{P:faith-flat}.
\end{remark}
\begin{proof}
Assume that $p$ is stably dominated by a $K$-definable function $h$ and let $a\models p|K$. By \cite[Section 7.5]{StableDomination}, every stable stably embedded set is definably isomorphic to a definable subset of the residue sort, so we may assume that $h$ is a definable function into a power of the residue field. 

Since $\Gamma(Ka)=\Gamma(K)$, by the Abhyankar inequality, 
\[\trdeg (k_{K(a)}/k_K)\leq \trdeg (K(a)/K).\]
On the other hand, since for every model $K\prec L$, and $b\in \mathcal{V}(\mathcal{O}_L)$, $r(b),h(b)\in \mathcal{V}_k(k_L)$, it follows that $r(a),h(a)\in k_{K(a)^{alg}}$ and thus\[\trdeg (k_K(r(a))/k_K)\leq \trdeg (k_{K(a)^{alg}}/k_K)=\trdeg (k_{K(a)}/k_K)\leq \trdeg (K(a)/K).\] 
Consequently, since $r_*p$ is generic and $\dim \mathcal{V}_K=\dim \mathcal{V}_k$, $h(a)$ is algebraic over $k_K(r(a))$. So if $a\models p|K$ and $r(a)\models r_*p|B$ for some $K\subseteq B$ then also $h(a)\models h_*p|B$ hence $a\models p|B$.

The in particular follows using Fact \ref{F:dim-generic=special}.
\end{proof}

Using Theorem \ref{T:Surjective} we will now show that, over models, generically stable types which are stably dominated by $r:\mathcal{V}(\mathcal{O})\to \mathcal{V}_k$ exist.

\begin{fact}\cite[Paragraph after Lemma 10.2]{StableDomination}\label{F:strong}
Let $K$ be a non-trivially valued algebraically closed field. If\[\trdeg (K(a)/K)=\trdeg (k_{K(a)}/k_K),\]
then $tp(a/K)$ is stably dominated and has a unique $K$-definable global extension.
\end{fact}

\begin{lemma}\cite[Lemma 6.7]{Metastable}\label{L:6.7fromMeta}
Let $K$ be a non-trivially valued algebraically closed field and $\mathcal{V}$ an integral separated scheme of finite type over $\mathcal{O}_K$. Every type $p=tp(a/K)$ which is 
\begin{enumerate}
\item concentrated on $\mathcal{V}(\mathcal{O})$ and for which
\item $r_*p$ is a generic type of $\mathcal{V}_k$ 
\end{enumerate}
is stably dominated, via $r:\mathcal{V}(\mathcal{O})\to \mathcal{V}_k$, and has a unique $K$-definable global extension. 

Moreover, every definable global type $q$ for which $(1)$ and $(2)$ hold, is $K$-definable.
\end{lemma}
\begin{proof}
Let $n:=\dim \mathcal{V}_K$. By Fact \ref{F:dim-generic=special}, $n=\dim \mathcal{V}_K=\dim \mathcal{V}_k$. 
 In general we have,
\[\trdeg (k_{K(a)}/k_K))+ \dim_\mathbb{Q}(\Gamma_{K(a)}/ \Gamma_K)\leq \trdeg (K(a)/K).\]
Since for every model $K\prec L$, and $b\in \mathcal{V}(\mathcal{O}_L)$, $r(b)\in \mathcal{V}_k(k_L)$, it follows that $r(a)\in k_{K(a)^{alg}}$ and so 
\[n=\trdeg (k_K(r(a))/k_K)\leq \trdeg (k_{K(a)^{alg}}/k_K)=\trdeg (k_{K(a)}/k_K).\]
Since
\[\trdeg (K(a)/K)\leq n,\]
$p=tp(a/K)$ is stably dominated and has a unique $K$-definable global extension by Fact \ref{F:strong}. The result follows using Proposition \ref{P:actually stab.dom}.

As for the moreover part, let $q$ be a definable global type as in the statement and assume it is definable over $K\subseteq L$, with $L$ an algebraically closed valued field. By the above, $q|K$ has a unique $K$-definable global extension, $\tilde{q}$. Let $a\models q|L$. Since $\tilde{q}$ is stably dominated over $K$ by $r_*q=r_*\tilde{q}$ via $r$ and $r(a)\models r_*q|L$, by stable domination $a\models \tilde{q}|L$. Thus $q$ and $\tilde{q}$ are both $L$-definable extensions of $q|L$, so by uniqueness they are equal and therefore $q$ is $K$-definable.  
\end{proof}

\begin{proposition}\label{P:existence of gen stable generic}
Let $K$ be a non-trivially valued algebraically closed field and let $\mathcal{V}$ be an integral separated scheme of finite type over $\mathcal{O}_K$. If $\mathcal{V}$ has an $\mathcal{O}_K$-point then there exists a $K$-definable type $p$ which is concentrated on $\mathcal{V}(\mathcal{O})$ and is stably dominated by $r_*p$ via the map $r:\mathcal{V}(\mathcal{O})\to \mathcal{V}_k$ where $r_*p$ is a generic type of $\mathcal{V}_k$.
\end{proposition}
\begin{proof}
This is a consequence of \Cref{T:Surjective} and Lemma \ref{L:6.7fromMeta}.
\end{proof}

\subsection{The Functor $\Phi_F$ For Affine Varieties}\label{ss:functor affine}
Let $F$ be a definably closed valued field. We will study the pair $(V,p)$, with $V$ affine, by defining a functor to schemes over $\mathcal{O}_F$.

\begin{definition}
We will say that a generically stable $F$-definable type is \emph{strictly based on $F$} if $\Gamma_{F(c)}=\Gamma_F$ for $c\models p|F$. 
\end{definition}
Notice that if $F$ has a divisible value group then every such type is strictly based on $F$.

\begin{definition}\label{D:the-category_OverF}
Denote by $\aGVar$ ("SPVar" for Stably Pointed Varieties) the category of pairs $(V,p)$ where $V$ is an affine variety over $F$ and $p$ is a Zariski dense generically stable type on $V$ definable over and strictly based on $F$. 

Morphisms $f:(V,p)\to (V',p')$ are morphisms of affine varieties over $F$, $f:V\to V'$, that push-forward $p$ appropriately, i.e. $f_*p=p'$. Let $\aGVar$ be the subcategory of affine varieties over $F$.
\end{definition}

Define a functor $\Phi^{\aff}_F:\aGVar\to \aSch$, which is given by
\[\Phi^{\aff}_F(V,p)=\Spec F[V]^p,\] where $F[V]^p:=\{f\in F[V]: (d_px)(\val (f(x))\geq 0)\}$ and $\aSch$ is the category of affine schemes over $\mathcal{O}_F$.
Since $F[V]^p$ is an $\mathcal{O}_F$-algebra, $\Spec F[V]^p$ is indeed a scheme over $\mathcal{O}_F$. As for functoriality,
assume that we have a  morphism \[\Theta:(V_1,p_1)\to (V_2,p_2).\] Let $f\in F[V_2]$ be such that \[(d_{p_2}x)(\val (f(x))\geq 0)\] and $c\models p_1|F$. Since $\Theta(c)\models p_2|F$, \[\val (f\circ \Theta)(c)\geq 0.\]

Let $(V,p)\in \aGVar$. For every $h\in F[V]$, let $F[V]_{(h)}$ be the corresponding localization. If \[(d_px)(\val (h(x))=0)\] then \[\left( F[V]_{h}\right)^p=\left(F[V]^p\right)_{h},\]
where $\left( F[V]_{(h)}\right)^p:=F[\Spec \left( F[V]\right)_{h}]^p$. We will write $F[V]_{h}^p$ for simplicity.
 
\begin{lemma}\label{L:affine-localization}
\begin{enumerate}
\item For any $0\neq r\in F[V]$ there exists $c\in F$ such that \[(d_px)(\val ((c^{-1}r)(x))=0).\]
\item For  $h\in F[V]$ s.t. $(d_px)(\val (h(x))=0)$, \[F[V]_{h}^p\otimes_{\mathcal{O}_F} F\cong F[V]_{h}\] as $F$-algebras.
\end{enumerate}
\end{lemma}
\begin{proof}

\begin{enumerate}
\item Let $a\models p|F$, since $p$ is strictly based on $F$ and thus $\Gamma_{F(a)}=\Gamma_F$, there exists $c\in F$ such that $\val (r(a))=\val (c)$. If $c=0$ then $r$ vanishes on a dense subset of $V(F)$ (i.e. $r=0$ in $F[V]$). Contradiction. So $\val ((c^{-1}r(a))=0$. 

\item Let $F[V]_{(h)}^p\otimes_{\mathcal{O}_F} F\to F[V]_{h}$ be the natural map given by $f\otimes a\mapsto af$.

Surjectivity: Let $r\in F[V]_{h}$, as in $(1)$, we may find $0\neq c\in F$ such that \[\val (c^{-1}r(a))\geq 0\] for $a\models p|F$ so $c^{-1}r\in F[V]_{h}^p$.

Injectivity: $F$ has no $\mathcal{O}_F$-torsion, so by flatness (Fact \ref{F:tors-free})
\[F[V]_{h}^p\otimes_{\mathcal{O}_F} F\hookrightarrow F[V]_{h}\otimes_{\mathcal{O}_F} F\cong F[V]_{h}.\]

\end{enumerate}
\end{proof}

It follows from the above lemma that \[\{\Spec F[V]_{f}:f\in F[V]^{p,0} \},\] where $F[V]^{p,0}=\{f\in F[V_i]^p:(d_px)(\val (f(x))=0)\}$, is a basis for the topology on $\Spec F[V]$.

The following is an easy observation.
\begin{remark}
Let $\Theta:(V_1,p_2)\to (V_2,p_2)$ be a morphism. Then $f\in F[V_2]^{p_2,0}$ if and only if $f\circ\Theta\in F[V_1]^{p_1,0}$.
\end{remark}

\begin{proposition}\label{P:geo-prop-functor}
Let $(V,p)\in\aGVar[F]$. Then $\Phi_F^{\aff}(V,p)$ is quasi-compact separated, integral and flat over $\mathcal{O}_F$. Furthermore, $\Phi_F^{\aff}(V,p)\times_{\mathcal{O}_F}F=V$.
\end{proposition}
\begin{proof}
Every affine scheme is quasi-compact and separated. The scheme $\Phi^{\aff}_F(V,p)$ is integral since $F[V]^p$ is an $\mathcal{O}_F$-subalgebra of $F[V]$. It is a flat since an $\mathcal{O}_F$-module is $\mathcal{O}_F$-torsion free if and only if it is flat (Fact \ref{F:tors-free}). The furthermore follows by Lemma \ref{L:affine-localization}.
\end{proof}

\subsubsection{Defectless Henselian Valued Fields and Descent}\label{ss:Descent}

If $F\subseteq F'$ are valued fields and $(V,p)\in \aGVar[F]$, we would like to find a connection between $\Phi^{\aff}_F(V,p)$ and $\Phi^{\aff}_{F'}(V_{F'},p)$. If $F$ is "well-behaved" we will prove that $\Phi^{\aff}_F(V,p)\times_{\mathcal{O}_F}\mathcal{O}_{F'}=\Phi^{\aff}_{F'}(V_{F'},p)$.

Let $F$ be an henselian valued field, and $L/F$ a finite extension of valued fields. By the fundamental equality 
\[[L:F]=d\cdot [\Gamma_L:\Gamma_F]\cdot [k_L:k_F],\] where $d$ is the defect of the extension. The extension $L/F$ is called \emph{defectless} if $d=1$ and $F$ is \emph{defectless} if $L/F$ is defectless for every finite extension $L/F$. We introduced this notion for henselian valued fields, but it exists for general valued fields (see \cite[Chapter 11]{kuhlmann} for more information).

The following are well known examples.
\begin{example}
\begin{enumerate}	
\item Every model of ACVF is trivially defectless.
\item Every henselian valued field of residue characteristic zero is a defectless field. 
\item Every spherically  complete field is a henselian defectless field.

\item For every prime $p$, $\mathbb{Q}_p$ and $\mathbb{F}_p((t))$ are henselian defectless valued fields.

\item There are examples of extensions with non-trivial defect (see \cite[Section 11.5]{kuhlmann}).
\end{enumerate}
\end{example}

Until the rest of the section let $F$ be an henselian perfect valued field with perfect residue field.

\begin{lemma}\label{L:finite-base-change-strictly}
Let $F$ be as above and $L/F$ a defectless finite extension of valued fields. Let $p$ be a generically stable $F$-definable type. If it is strictly based on $F$ then it is also strictly based on $L$.
\end{lemma}
\begin{proof}
Let $c\models p|L$. Since $k_F$ is perfect, by \cite[Proposition 8.19]{StableDomination} $k_{F(c)}$ is linearly disjoint from $k_L$ over $k_F$ and thus $[k_{L(c)}:k_{F(c)}]=[k_L:k_F]$. Since $p$ is strictly based on $F$, i.e. $\Gamma_{F(c)}=\Gamma_F$, \[[\Gamma_L:\Gamma_F]\leq [\Gamma_{L(c)}:\Gamma_{F(c)}],\] and obviously $[L(c):F(c)]\leq [L:F]$. Thus by the fundamental inequality \[[\Gamma_{L(c)}:\Gamma_{F(c)}]\leq \frac{[L(c):F(c)]}{[k_{L(c)}:k_{F(c)}]}\leq \frac{[L:F]}{[k_L:k_F]}=[\Gamma_L:\Gamma_F]\] and consequently $\Gamma_{L(c)}=\Gamma_L$.
\end{proof}

\begin{fact}\label{F:val-basis-for hen-defect}\cite[Lemma 6.17]{kuhlmann}
Let $F$ be an henselian valued field and $L/F$ a finite extension. If $L/F$ is defectless then $L/F$ admits a valuation basis, i.e. there exists a basis $c_1,\dots,c_n$ of $L/F$ such that for every $a_1,\dots, a_n\in F$
\[\val (a_1c_1+\dots a_nc_n)=\min_{i}\{\val (a_ic_i)\}.\]
\end{fact}

In fact one may choose this valuation basis to be a \emph{standard valuation independent set}, i.e. of the form $\{b_i'b_j''\}_{i,j}$ where the values $\{\val (b_i')\}_i$ lie in distinct cosets of $\Gamma_L$ modulo $\Gamma_F$ and $\{b_j''\}_j$ are elements of $0$ valuation whose residues are $k_F$-linearly independent.

\begin{definition}
Let $L_1/F$ and $L_2/F$ be two extensions of valued fields. We say that $L_1$ is \emph{valuation disjoint} from $L_2$ over $F$ if every standard valuation independent set of $L_1/F$ is also a standard valuation independent set of $L_2L_1/L_2$.

\end{definition}

\begin{fact}\cite[Lemma 2.19]{kuhlmann2}\label{F:val-disj-fvk}
Let $L_1/F$ and $L_2/F$ be two extension of valued fields. Then $L_1/F$ is valuation disjoint from $L_2/F$ if and only if
\begin{enumerate}
\item $\Gamma_{L_1}\cap \Gamma_{L_2}=\Gamma_F$, and
\item $k_{L_1}$ is linearly disjoint from $k_{L_2}$ over $k_F$.
\end{enumerate}
\end{fact}

\begin{lemma}\label{L:val-disj}
Let $F$ be as above, $p$ a generically stable $F$-definable type which is strictly based on $F$ and $L/F$ an extension of valued fields. Then $L$ is valuation disjoint from $F(c)$ over $F$, for $c\models p|L$.
\end{lemma}
\begin{proof}
We use Fact \ref{F:val-disj-fvk}. Since $p$ is strictly based on $F$, $\Gamma_{F(c)}\cap \Gamma_L=\Gamma_F\cap \Gamma_L=\Gamma_F$.

Recall that for a valued field field $A$, $k(A):=\dcl(A)\cap k(\mathbb{U})$. By \cite[Proposition 8.19]{StableDomination}, $k(\acl(F(a))$ and $k(L)$ are linearly disjoint over $k(F)$. Also, since $k_F$ is perfect, $k_F=k(F)$ and the result follows.
\end{proof}

Recall the following result:

\begin{lemma}\label{L:fraction field=generic point}
Let $(V,p)\in \aGVar$. Then $F(V)$ (the field of rational functions on $V$) is a valued field with valuation \[F(V)\ni f\mapsto \val (f(c))\] where $c\models p$.
\end{lemma}
\begin{proof}
Let $c\models p$. Since $p$ is dense in $V$, for $f(x)\in F[V]$, $\val (f(c))=\infty$ if and only if $f=0$ (in $F[V]$). So the map given is indeed a valuation on $F(V)$ and extends the one given on $F$.
\end{proof}

Combining the above we get our desired result for defectless finite extensions:

\begin{proposition}\label{P:descent}
Let $F$ be as above, $(V,p)\in\aGVar$ and $L/F$ be a defectless finite extension of valued fields.
Then \[\Phi^{\aff}_F(V,p)\times_{\mathcal{O}_F}\mathcal{O}_L=\Phi^{\aff}_L(V_L,p).\]
\end{proposition}
\begin{proof}
If $F$ is trivially valued there is nothing to prove.

%By functoriality and the remark after Proposition \ref{P:functor exists}, it is enough to prove the statement for affine $V$. 
Let $\mathcal{B}=(v_1,\dots,v_n)$ be a standard valuation basis for $L/F$ and $c\models p|L$. By Lemma \ref{L:val-disj}, $\mathcal{B}$ is also a standard valuation basis for $L(c)/F(c)$. Recall that $F(V)\simeq F(c)$ and $L(V)\simeq L(c)$. Define \[F(V)^p=\{f\in F(V): \val (f(c))\geq 0\},\] and similarly $L(V)^p$, they are the valuation rings of $F(V)$ and $L(V)$ with respect to the valuation given by Lemma \ref{L:fraction field=generic point}.

\begin{claim}
$L(V)^p=F(V)^p\otimes_{\mathcal{O}_F}\mathcal{O}_L$.
\end{claim}
\begin{claimproof}
One inclusion is obvious, for the other let $f\in L(V)^p$. Since $\mathcal{B}$ is a basis there exist $f_1,\dots,f_n\in F(V)$ such that \[f=\sum_{i=1}^n f_i\cdot v_i.\] 
Since $\mathcal{B}$ is a valuation basis and $f\in L(V)^p$, $\val (f_i\cdot v_i)\geq 0$ for every $i$. The type $p$ is strictly based on $F$ and hence for every $i$ with $f_i\neq 0$ we may choose $a_i\in F^\times$ such that $\val (f_i)=\val (a_i)$ and write
\[f=\sum_{i=1}^n (a_i^{-1}f_i)\cdot (a_iv_i).\]
Thus $a_i^{-1}f_i\in F(V)^p$ and $a_iv_i\in \mathcal{O}_L$.
\end{claimproof}

Since $L[V]=F[V]\otimes_F L$ and $\mathcal{B}$ is a basis the Claim implies that \[L[V]^p=F[V]^p\otimes_{\mathcal{O}_F}\mathcal{O}_L.\]
\end{proof}

\begin{corollary}\label{C:descent-for-alg}
Let $(V,p)\in\aGVar$ where $F$ is as above.
If in addition $F$ is defectless, then $\Phi^{\aff}_F(V,p)\times_{\mathcal{O}_F}\mathcal{O}_L=\Phi^{\aff}_L(V_L,p)$ for every algebraic extension of valued fields $L/F$.
\end{corollary}
\begin{proof}
Every $f\in L[V]^p$ lies in $L'[V]^p$ for some finite extension of valued fields $L'/F$.
\end{proof}

We will now show that we may descend between models. 

Recall that a valued field $C$ is \emph{maximally complete} if it has no immediate proper extension. By Zorn's Lemma every valued field has an immediate maximally complete extension (\cite[Theorem 8.22]{kuhlmann}). A maximally complete immediate extension of an algebraically closed valued field is also algebraically closed. Thus, by quantifier elimination, for $K\models$ACVF there exists $K\prec K_1$ with $K_1$ maximally complete.

\begin{fact}\cite[Proposition 12.1]{StableDomination}\label{F:val-basis-over-max-complete}
Let $C\leq A$ be an extension of non-trivially valued fields, with $C$ maximally complete, and $V$ be a finite dimensional $C$-vector subspace of $A$. Then $V$ admits a valuation basis.
\end{fact}
%
%The following helpful lemma is well known.
%\begin{lemma}\label{L:lin-ind-min}
%Let $F\subseteq A$ be an extension of valued fields and $f_1,\dots, f_n\in \mathcal{O}_A$. Then $\bar{f}_1,\dots,\bar{f}_n$ are $k_F$ linearly independent if and only if $f_1,\dots,f_n$ are valuation independent, i.e. for every $a_1,\dots, a_n\in F$, \[\val (a_1f_1+\dots+a_nf_n)=\min_i \{\val (a_if_i)\}.\]
%\end{lemma}
%\begin{proof}
%Assume that $\bar{f}_1,\dots,\bar{f}_n$ are $k_F$ linearly independent. Let $a_1,\dots, a_n\in F$ and assume that $a_1\neq 0$ and has minimum valuation. By assumption, necessarily \[\val \left( \sum_i \frac{a_i}{a_1}f_i\right)=0.\]
%The other direction is straightforward.
%\end{proof}

\begin{proposition}\label{P:desc-between-models}
Let $K\subseteq L$ be an extension of models of ACVF and $(V,p)\in\aGVar[K]$. Then $\Phi^{\aff}_K(V,p)\times_{\mathcal{O}_K}\mathcal{O}_L=\Phi^{\aff}_L(V_L,p)$.
\end{proposition}
\begin{proof}
%It is enough to prove the statement for affine $V$, i.e. 
We will show that 
\[K[V]^p\otimes_{\mathcal{O}_K}\mathcal{O}_L=L[V]^p.\] Since $K[V]^p\otimes_{\mathcal{O}_K}\mathcal{O}_L\subseteq L[V]^p$, we show the other direction.

Let $L\prec K_1$ be a maximally complete extension and let $c\models p|K_1$. 

Assume that $L[V]=L[\bar{X}]/I$. Let $H_d$ be the vector space of polynomials of degree $\leq d$, $I_d=H_d\cap I$ and $U_d=H_d/I_d$. It is a subvector space of $L[\bar{X}]/I$. Denote by $L[c]_{\leq d}$ the $L$-vector spaces of polynomials of degree $\leq d$ in $c$. This may be identified with the vector space $U_d$. Similarly, define $K[c]_{\leq d}$ and thus \[K[c]_{\leq d}\otimes_K L=L[c]_{\leq d}.\] 

\begin{claim}
$K[c]_{\leq d}$ has a valuation basis over $K$.
\end{claim}
\begin{claimproof}
By Fact \ref{F:val-basis-over-max-complete} there exists a valuation basis $f_1,\dots,f_n$ of $K_1[c]_{\leq d}$, i.e. for any $a_1,\dots, a_n\in K_1$
\[(d_px)\left( \val \left( \sum_{i=1}^na_if_i(x)\right)=\min\{\val (a_1f_1(x)),\dots,\val (a_nf_n(x))\}\right).\]
Since the degrees of the basis elements and polynomials in question are bounded by $d$ the existence of such generators is a first order sentence so by model completeness we may assume that such exist in $K$.
\end{claimproof}

Let $f_1(c),\dots,f_n(c)$ be a valuation basis of $K[c]_{\leq d}$ over $K$. Since $p$ is strictly based on $K$ we may assume that $\val (f_i(c))=0$ for all $i$. As in the proof of the claim, since the fact that the $f_i(c)$ form a valuation basis is contained in the type, they also form a valuation basis over $L$.  

%By Lemma \ref{L:lin-ind-min} $\overline{f_1(c)},\dots,\overline{f_n(c)}$ are $k_K$ linearly independent and since $k_{K(c)}/k_K$ is linearly disjoint from $k_L/k_K$ (\cite[Proposition 8.19]{StableDomination}), they are also linearly independent over $k_L$. Again, by Lemma \ref{L:lin-ind-min}, $f_1(c),\dots,f_n(c)$ are valuation independent over $L$.

Let $f\in L[V]^p$ of degree $\leq d$. By the above there exist $a_1,\dots,a_n\in L$ such that $f(c)=\sum_i a_if_i(c)$, and since the $f_i(c)$ form a valuation basis and $\val f_i(c)=0$, necessarily $a_i\in \mathcal{O}_L$ as needed.
\end{proof}

\subsubsection{Products}

Using the descent results from the previous section, we show that the functor $\Phi^{\aff}_F$, for nice enough $F$, commutes with finite products.

The following fact is essentially Lemma 12.4 in \cite{StableDomination}, the proof of the Lemma gives the rest of the statements

\begin{fact}\cite[Lemmas 12.4]{StableDomination}\label{F:prod-from-stabdom}
Let $C$ be a maximally complete algebraically closed valued field, and $C\subseteq A,B\subseteq \mathbb{U}$ two valued field extensions. Assume that $\Gamma(C)=\Gamma(A)$ and that $k(A)$ and $k(B)$ are linearly disjoint over $k(C)$. For any $b_1,\dots,b_n\in B$ there exist $b_1^\prime,\dots, b_k^\prime$ ($k\leq n$) in the $C$-vector space spanned by $\{b_1,\dots,b_n\}$ satisfying $\val(b_j^\prime)=0$ for all $1\leq j\leq k$ and that 
\begin{enumerate}
\item for every $a_1,\dots,a_k\in A$,
\[\val \left(\sum_{i=1}^k a_ib_i^\prime\right)=\min_j \{\val (a_j)\};\]
\item for any $a_1\dots,a_n\in A$ there exist $a^\prime_1,\dots, a_k^\prime$ in the $C$-vector space spanned by $\{a_1,\dots,a_n\}$ such that in $A\otimes_C B$ we have $\sum_{i=1}^n a_i\otimes b_i=\sum_{j=1}^k a^\prime_j\otimes b_j^\prime$ and \[\val \left(\sum_{i=1}^n a_ib_i\right)=\min_j \{\val (a^\prime_j)\}.\]
\end{enumerate}
\end{fact}

\begin{proposition}\label{P:Cor_from_Lem12.4}
Fact \ref{F:prod-from-stabdom} holds for $K\subseteq K(c), K(d)$ where $K$ is a maximally complete algebraically closed valued field and $(c,d)\models (p\otimes q)|K$ for generically stable $K$-definable $p$ and $q$.

\end{proposition}
\begin{proof}
By using \cite[Proposition 8.19]{StableDomination}, the assumptions of Fact \ref{F:prod-from-stabdom} hold.
\end{proof}

\begin{proposition}\label{P:prod-two-affines}
Let $F$ be a defectless henselian perfect valued field with perfect residue field and $(V_1,p_1),(V_2,p_2)\in\aGVar$, with $p_1\otimes p_2$ strictly based on $F$. Then \[F[V_1\times_F V_2]^{p_1\otimes p_2}=F[V_1]^{p_1}\otimes_{\mathcal{O}_F}F[V_2]^{p_2}.\]
\end{proposition}
\begin{remark}
\begin{enumerate}
\item Recall that if $F$ is a model of ACVF then $p_1\otimes p_2$ is automatically strictly based on $F$.
\item Note that if $R_1, R_2$ are $F$-algebras and $R_1',R_2'$ torsion-free $\mathcal{O}_F$-subalgebras, respectively, then \[R_1'\otimes_{\mathcal{O}_F} R_2' \hookrightarrow R_1\otimes_{\mathcal{O}_F} R_2.\]
Indeed, for $i=1,2$, $R_i$ and $R_i'$ are $\mathcal{O}_F$-torsion free and hence flat $\mathcal{O}_F$-modules by Fact \ref{F:tors-free}, so we have the following injections:
\[R_1'\otimes_{\mathcal{O}_F} R_2' \hookrightarrow R_1'\otimes_{\mathcal{O}_F} R_2\hookrightarrow R_1\otimes_{\mathcal{O}_F} R_2,\]
thus we may identify $R_1'\otimes_{\mathcal{O}_F} R_2'$ with its image in $R_1\otimes_{\mathcal{O}_F} R_2$.
Notice that 
\[R_1\otimes_{\mathcal{O}_F} R_2
= R_1\otimes_F R_2\] as $\mathcal{O}_F$-algebras. Specifically we identify $F[V_1]^{p_1}\otimes_{\mathcal{O}_F}F[V_2]^{p_2}$ as an $\mathcal{O}_F$-subalgebra of $F[V_1]\otimes_{\mathcal{O}_F}F[V_2]=F[V_1]\otimes_F F[V_2]$.
\end{enumerate} 
\end{remark}
\begin{proof}
The proof is a slight generalization, due to descent, but mostly identical to the an argument given in \cite[Proposition 6.11]{Metastable}. 
If $F$ is trivially valued there is nothing to prove. Otherwise, Let $F\subseteq K$ be a maximally complete algebraically closed valued field extension. 

Let $\sum_{i=1}^m f_i\otimes g_i\in K[V_1\times_K V_2]^{p_1\otimes p_2}$ and let $(c,d)\models (p_1\otimes p_2)|K$. By Proposition \ref{P:Cor_from_Lem12.4}, we may find $f_i^\prime\in K[V_1]$ and $g_i^\prime\in K[V_2]^{p_2,0}$ such that
\[\sum_{i=1}^{m}f_i\otimes g_i=\sum_{i=1}^{m^\prime}f_i^\prime\otimes g_i^\prime\] and
\[0\leq \val \left( \sum_{i=1}^{m^\prime} f_i^\prime (c)\cdot g_i^\prime(d)\right)=\min_i\{\val (f_i^\prime(c))\}.\]
In particular $\sum_{i=1}^{m}f_i\otimes g_i=\sum_{i=1}^{m^\prime}f_i^\prime\otimes g_i^\prime\in K[V_1]^{p_1}\otimes_{\mathcal{O}_K}K[V_2]^{p_2}$.

By Corollary \ref{C:descent-for-alg} and Proposition \ref{P:desc-between-models}

\[F[V_1\times_F V_2]^{p_1\otimes p_2}\otimes_{\mathcal{O}_F}\mathcal{O}_K=\left( F[V_1]^{p_1}\otimes_{\mathcal{O}_F}F[V_2]^{p_2} \right)\otimes_{\mathcal{O}_F}\mathcal{O}_K.\] Since $\mathcal{O}_F\subseteq \mathcal{O}_K$, and $\mathcal{O}_K$ is an integral domain, $\mathcal{O}_K$ is flat over $\mathcal{O}_F$ by Fact \ref{F:tors-free}.
The result follows by faithfully flat descent (see \cite[Proposition 14.51]{gortz}).
\end{proof}

\subsection{$\Phi$ and Model Theoretic Properties}\label{ss:the_functor}
Recall that we denote by $\mathbb{K}$ the valued field sort the monster model $\mathbb{U}$ and that $K$, unless stated otherwise, usually denotes a non-trivially valued algebraically closed field.

By Proposition \ref{P:desc-between-models}, if $(V,p)\in\GVar[K]$ then $\Phi^{\aff}_L(V_L,p)=\Phi^{\aff}_K(V,p)\times_{\mathcal{O}_K}\mathcal{O}_L$ for any algebraically closed valued field $L$ containing $K$. Consequently, it is relatively harmless to drop the subscript $K$ from $\Phi^{\aff}_K$. 

Recall Section \ref{ss:Pro-def-grps}, since $\Phi^{\aff}(V,p)$ is quasi-compact, it is an inverse limit of schemes of finite type over $\mathcal{O}_K$ (necessarily indexed by a small set, with respect to the monster model), $\Phi^{\aff}(V,p)_K:=\Phi^{\aff}(V,p)\times_{\mathcal{O}_K} K$ is a pro-definable set, the emphasis here was that we can take the index set to be small.  
Thus, for $\mathcal{V}:=\Phi^{\aff}(V,p)$, viewing $\mathcal{V}_K$ a pro-definable set, $\mathcal{V}(\mathcal{O})\subseteq \mathcal{V}_K$ is a pro-definable subset. By Proposition \ref{P:geo-prop-functor}, $\mathcal{V}_K\cong V$ as varieties over $K$. Under this isomorphism $p$ is concentrated on $\mathcal{V}_K$. We will identify $p$ with its image under this isomorphism.

In the previous sections we described $\Phi^{\aff}$ (more specifically $\Phi^{\aff}_K$). In the following section, we give some model theoretic properties of $\Phi^{\aff}(V,p)$ and conclude that after restricting the codomain, it is fully faithful.

We will first need the following definition, whose origin can be seen in \cite[Proposition 6.9]{Metastable}.

\begin{definition}\label{D:mmp}
Let $K$ be model of ACVF, $\mathcal{V}$ be an affine scheme over $\mathcal{O}_K$, $p$ a $K$-definable type concentrated on $\mathcal{V}(\mathcal{O})$ (so $\mathcal{V}(\mathcal{O})\neq \emptyset$). We say that $\mathcal{V}$ has \emph{the maximum modulus principle with respect to $p$} (written, the mmp w.r.t. $p$) if for every regular function $f$ on $\mathcal{V}_\mathbb{K}=\mathcal{V}\times_{\mathcal{O}_K}\mathbb{K}$ there is some $\gamma_f\in \Gamma$ such that \[(d_px)(\val (f(x))=\gamma_f)\] and for any $h\in \mathcal{V}(\mathcal{O})$ \[\val (f(h))\geq \gamma_f.\]

If $\mathcal{V}$ is a quasi-compact separated scheme over $\mathcal{O}_K$ (not necessarily affine) then we will say that it has the mmp w.r.t. $p$ if $\mathcal{V}$ has an affine open cover, consisting of schemes over $\mathcal{O}_K$, \[\mathcal{V}=\bigcup_{i=1}^n \mathcal{U}_i\] such that for every $i$ with $\mathcal{U}_i(\mathcal{O})\neq \emptyset$, $p$ is concentrated on $\mathcal{U}_i(\mathcal{O})$ and has the mmp w.r.t. $p$.
\end{definition}

\begin{remarkcnt}\label{R:after mmp}
\begin{enumerate}
\item The two definitions coincide for affine schemes over $\mathcal{O}_K$.
\item By definition, $\mathcal{V}$ has the mmp w.r.t. $p$ if and only if $\mathcal{V}_{\mathcal{O}}=\mathcal{V}\times_{\mathcal{O}_K}\mathcal{O}$ has the mmp w.r.t. $p$.
\item If $\mathcal{V}$ is of finite type over $\mathcal{O}_K$, then $(2)$ would also be true if in the definition we only consider regular functions on $\mathcal{V}_K$.
\item By quantifier elimination in the language $L_{\text{div}}$, every formula is equivalent to a boolean combination of formulas of the sort $\val (f(x))\leq \val (g(x))$, so if $\mathcal{V}$ has the mmp w.r.t. $p$ and w.r.t. $q$ then $p=q$.
\end{enumerate}
\end{remarkcnt}

\begin{lemma}\label{L:implications-of-mmp}
Let $\mathcal{V}$ be a quasi-compact separated integral scheme over $\mathcal{O}_K$ and $p$ a $K$-definable type concentrated on $\mathcal{V}(\mathcal{O})$. If $\mathcal{V}$ has the mmp w.r.t. $p$ then $p$ is generically stable and Zariski dense in $\mathcal{V}_K$. Furthermore, $r_*p$ is Zariski dense in $\mathcal{V}_k$ and hence $\mathcal{V}_k$ is (geometrically-)irreducible, where $r:\mathcal{V}(\mathcal{O})\to \mathcal{V}_k$.
\end{lemma}
\begin{proof}
We may assume that $\mathcal{V}$ is affine. Since for every regular function on $f$ on $\mathcal{V}_\mathbb{K}$ there is some $\gamma_f\in \Gamma$ such that 
\[(d_px)(\val (f(x))=\gamma_f),\] by quantifier elimination, $p$ is orthogonal to $\Gamma$ and hence generically stable.

Let $f$ be a regular function on $\mathcal{V}_K$ and let $c\models p|K$. If $f(c)=0$ then by the mmp, $f(h)=0$ for every $h\in \mathcal{V}(\mathcal{O})$. Also, $\mathcal{V}(\mathcal{O})$ is Zariski dense in $\mathcal{V}_K$ by Proposition \ref{P:O-points-are-dense} and thus $f\equiv 0$ on $\mathcal{V}_K$.

As for the furthermore, let $\bar{f}\neq 0$ be a regular function on $\mathcal{V}_k$, $a\in \mathcal{V}_k$ satisfying $\bar{f}(a)\neq 0$, and $c\models p|K$. The regular function $\bar{f}$ arises from some global section on $\mathcal{V}$. By Theorem \ref{T:Surjective}, there exists $b\in \mathcal{V}(\mathcal{O})$ with $r(b)=a$ and consequently, since $\bar{f}(a)\neq 0$, $\val(f(b))=0$. By the mmp w.r.t. $p$, $\val (f(c))=0$ and hence \[\bar{f}(r(c))\neq 0,\] as needed.
\end{proof}

\begin{lemma}\label{L:V_Kto V dominant}
Let $F$ be a valued field and $\mathcal{V}$ be a scheme over $\mathcal{O}_F$. If $\mathcal{V}$ is integral and has an $\mathcal{O}_F$-point then $\mathcal{V}_F\to \mathcal{V}$ is dominant. 
\end{lemma}
\begin{proof}
Since $\Spec F\to \Spec \mathcal{O}_F$ is quasi-compact and dominant, and since, by Proposition \ref{P:faith-flat}, $\mathcal{V}$ is flat over $\mathcal{O}_F$,  we may base-change (see \cite[Exercise 14.14]{gortz}).
\end{proof}

\begin{proposition}\label{P:mmp iff for affine cover}
Let $\mathcal{V}$ be a quasi-compact separated integral scheme over $\mathcal{O}_K$ with an $\mathcal{O}$-point and $p$ a $K$-definable type concentrated on $\mathcal{V}(\mathcal{O})$. Then the following are equivalent:
\begin{enumerate}
\item $\mathcal{V}$ has the mmp w.r.t. $p$;
\item for every basic open affine subscheme $\mathcal{U}\subseteq \mathcal{V}_\mathcal{O}:=\mathcal{V}\times_{\mathcal{O}_K}\mathcal{O}$ with $\mathcal{U}(\mathcal{O})$ non-empty, $p$ is concentrated on $\mathcal{U}(\mathcal{O})$ and $\mathcal{U}$ has the mmp w.r.t. $p$;
\item for every open subscheme $\mathcal{U}\subseteq \mathcal{V}_{\mathcal{O}}$ with $\mathcal{U}(\mathcal{O})$ non-empty, $p$ is concentrated on $\mathcal{U}(\mathcal{O})$ and $\mathcal{U}$ has the mmp w.r.t. $p$.
\end{enumerate}
\end{proposition}
\begin{remark}
As in Remark \ref{R:after mmp}, if $\mathcal{V}$ is of finite type over $\mathcal{O}_K$ one may restrict to open subschemes of $\mathcal{V}$.
\end{remark}
\begin{proof}
By Remark \ref{R:after mmp}(2), we may replace $\mathcal{V}$ by $\mathcal{V}_\mathcal{O}$.

$(3)\implies (2)$: This is straightforward and follows for the definition.

$(2)\implies (1)$: We may assume that $\mathcal{V}$ is affine. Let $f$ be a regular function on $\mathcal{V}_\mathbb{K}$, and let $a\in \mathcal{V}(\mathcal{O})$ and thus there exists a basic open affine subscheme $D_{\mathcal{V}}(g)\subseteq \mathcal{V}$ with $a\in D_{\mathcal{V}}(g)(\mathcal{O})$. In particular $\val(g(a))=0$. By Lemma \ref{L:V_Kto V dominant}, seeing $f$ as regular function on $D_{\mathcal{V}}(g)_\mathbb{K}$,
\[(d_px)(\val(f(a))\geq val(f(x)),\] and this is also true if we consider $f$ as a regular function on all of $\mathcal{V}_\mathbb{K}$.

$(1)\implies (3)$: By applying direction $(2)\implies (1)$ on $\mathcal{U}$, it is enough to show $(2)$. Thus, we may assume that $\mathcal{V}$ is affine.
Let $D_{\mathcal{V}}(f)\subseteq \mathcal{V}$ with $D_{\mathcal{V}}(f)(\mathcal{O})$ non-empty.

Note that $D_{\mathcal{V}}(f)_\mathbb{K}$ is a basic open subset of $\mathcal{V}_\mathbb{K}$, and by Lemma \ref{L:V_Kto V dominant}, every regular function on it has the form $h/f^k$ for some $h$ regular on $\mathcal{V}_\mathbb{K}$. There exists $a\in D_{\mathcal{V}}(f)(\mathcal{O})$, in particular $\val (f^k(a))=0$ for every $k\in\mathbb{N}$. Since $a\in \mathcal{V}(\mathcal{O})$ as well, by the maximum modulus principle it is also true that $(d_px)(\val (f(x))=0)$, hence $p$ is concentrated on $D_{\mathcal{V}}(f)(\mathcal{O})$. Similarly if \[(d_px)(\val (h(x)/f^k(x))=\alpha)\] then \[\val (h(a)/f^k(a))\geq \alpha.\] Indeed, since the denominators have zero valuations, it follows from the maximum modulus principle for $\mathcal{V}$.
\end{proof}

\begin{example}\label{E:GL_n-mmp}
Denote by $\mathrm{GL}_n$ the (group-)scheme of invertible matrices over $\mathcal{O}_K$. Consider it as a definable subset of $\mathbb{K}^{n^2+1}$. 
Set $p:=\varphi_*p_\mathcal{O}^{n^2}$, where $\varphi$ is the definable map \[\bar{x}\mapsto (\bar{x},\frac{1}{\det(\bar{x})}).\] Obviously $p$ is concentrated on $\mathrm{GL}_n(\mathcal{O})$ and using Example \ref{E:p_O-minimal-val} one easily sees that $\mathrm{GL}_n$ has the mmp w.r.t. $p$.
\end{example}

We conclude with a summary of some properties of $\Phi$.

\begin{lemma}\label{L:same-type-iff-dense}
Let $(V,p)\in \aGVar[K]$ and $q$ be a generically stable $K$-definable type concentrated on $V$. Then $q$ is concentrated on every $D_{\mathcal{V}}(f)(\mathcal{O})$, where $\mathcal{V}=\Phi_K^{\aff}(V,p)$ and $D_{\mathcal{V}}(f)(\mathcal{O})\neq \emptyset$, if and only if $p=q$.
\end{lemma}
\begin{proof}
We will first show that $p$ is concentrated on each such $D_{\mathcal{V}}(f)(\mathcal{O})$ where $f\in K[V]^p$, with $D_{\mathcal{V}}(f)(\mathcal{O})\neq\emptyset$. By the latter assumption, there exists $a\in\mathcal{V}(\mathcal{O})$ with $\val (f(a))=0$. Consequently, if $b\models p|K$ and $c\in K$ is such that $\val(c^{-1}f(b))=0$ then, since $c^{-1}f\in K[V]^p$, $\val(c^{-1}f(a))\geq 0$ and so $\val(f(b))=0$ which gives that $p$ is concentrated on $D_{\mathcal{V}}(f)(\mathcal{O})$.

For the other direction, let$f$ be a non-zero regular function on $V$ and $c\in K$ such that \[(d_px)(\val (c^{-1}f(x))=0).\] By assumption we also have \[(d_qx)(\val (c^{-1}f(x))=0).\]
The result, now, follows from quantifier elimination in the language $L_{\text{div}}$, since in this language every formula is equivalent to a boolean combination of formulas of the sort $\val ((f(x))\leq \val ((g(x))$.
\end{proof}

\begin{proposition}\label{P:properties of Phi(V,p)}
Let $(V,p)\in \aGVar[K]$ and $\mathcal{V}=\Phi^{\aff}(V,p)$. Then $\mathcal{V}$ and $p$ enjoy the following properties
\begin{enumerate}
\item $\mathcal{V}$ is quasi-compact separated, integral and flat over $\mathcal{O}_K$.
\item $\mathcal{V}\times_{\mathcal{O}_K}K=V$.
\item $\mathcal{V}_{\mathcal{O}_L}:=\mathcal{V}\times_{\mathcal{O}_K}\mathcal{O}_L$ has the maximum modulus principle w.r.t. $p$, for any model $L$ containing $K$, in particular $p$ is concentrated on $\mathcal{V}(\mathcal{O})$.
\item Let $r:\mathcal{V}(\mathcal{O})\to \mathcal{V}_k$. Then $r_*p$ is Zariski dense in $\mathcal{V}_k$ and hence $\mathcal{V}_k$ is (geometrically-)irreducible.
%\item $p$ is stably dominated by $r_*p$ via $r$.
\end{enumerate}
\end{proposition}
\begin{proof}
\begin{enumerate}
\item See Proposition \ref{P:geo-prop-functor}.
\item See Proposition \ref{P:geo-prop-functor}.

\item After base-changing, we may replace $V$ by $V_{\mathbb{K}}$ and so $\mathcal{V}$ by $\mathcal{V}_\mathcal{O}$. By definition for any $a\models p|L$ and $f\in L[V]^p$, where $L$ is small model, $\val (f(a))\geq 0$. Consequently $p$ is concentrated on $\mathcal{V}(\mathcal{O})$.

Let $f$ be a regular function on $V=\mathcal{V}_\mathbb{K}$ and $L$ a small model over which $f$ is defined. Since $p$ is generically stable and $L$ is a model,
\[(d_px)(\val (f(x))=\val(c))\]
for some $c\in L$.

As a result, $c^{-1}f\in L[V]^p$, and so by definition, for every $h\in \mathcal{V}(\mathcal{O})$, $\val (c^{-1}f(h))\geq 0$ and thus
\[\val (f(h))\geq \val(c).\]

\item Lemma \ref{L:implications-of-mmp}.
%By definition of the maximum modulus principle, there exists an affine open cover $\mathcal{V}=\bigcup_{i=1}^n \mathcal{V}_i$ with $p$ concentrated on each $\mathcal{V}_i(\mathcal{O})$ (by the Remark after Proposition \ref{P:functor exists}, each of them is non-empty). So $(\mathcal{U}_i)_k\cap (\mathcal{U}_j)_k\neq \emptyset$ for all $1\leq i,j\leq n$. It remains to show that $r_*p$ is Zariski dense in each $(\mathcal{U}_i)_k$. As a result, we may assume that $\mathcal{V}$ is affine.
%
%Let $\bar{f}\neq 0$ be a regular function on $\mathcal{V}_k$, and $c\models p|K$ for some $F\subseteq K$ over which $f$ is defined. The regular function $\bar{f}$ arises from some $f\in K[V]^p$. By the maximum modulus principle with respect to $p$, $\val (f(c))=0$ and hence \[\bar{f}(r(c))\neq 0,\] as needed.

%\item Proposition \ref{P:actually stab.dom}.
\end{enumerate}
\end{proof}

%
%\begin{lemma}\label{L:preservesifsurj}
%Let $(V_1,p_1),(V_2,p_2)\in \GVar$ and $\Theta:\mathcal{V}_1\to \mathcal{V}_2$ a morphism, where $\mathcal{V}_i=\Phi(V_i,p_i)$ for $i=1,2$. If for every affine open subscheme $\mathcal{U}\subseteq \mathcal{V}_2$ with $\mathcal{U}(\mathcal{O})\neq \emptyset$ also $(\Theta^{-1}(\mathcal{U}))(\mathcal{O})\neq \emptyset$ then $\Theta_*p_1=p_2$.
%\end{lemma}
%\begin{proof}
%Let $\mathcal{U}\subseteq \mathcal{V}_2$ be an affine open subscheme with $\mathcal{U}(\mathcal{O})\neq \emptyset$. Since $\Theta^{-1}(\mathcal{U})(\mathcal{O})\neq \emptyset$, by Lemma \ref{L:mmp iff for affine cover} $p_1$ is concentrated on $\Theta^{-1}(\mathcal{U})(\mathcal{O})$ and $p_2$ on $\mathcal{U}(\mathcal{O})$. 
%
%Let $c_i\models p_i|F$, for $i=1,2$. 
%By Lemma \ref{L:mmp iff for affine cover} $\Theta^{-1}(\mathcal{U})$ has the mmp w.r.t. $p_1$. Let $c_2'\in \Theta^{-1}(c_2)$. For every regular function $f$ on $\mathcal{U}_K$, by the maximum modulus principle of $\mathcal{U}$ and $\Theta^{-1}(\mathcal{U})$ w.r.t. $p_2$ and $p_1$, respectively,
%\[valf(\Theta(c_1))\leq valf(\Theta(c_2'))=valf(c_2)\leq valf(\Theta(c_1)).\]
%Thus $\Theta_*p_1=p_2$.
%\end{proof}

We conclude this section by showing that for affine varieties, after restricting the codomain, $\Phi_K$ is fully faithful.

\begin{remark} 
Note that we really need to restrict the codomain, i.e. the functor $\Phi_K$ on affine varieties is not full. Recall that $\Phi_K(\mathbb{A}^1_K,p_{\mathcal{O}})=\mathbb{A}^1_{\mathcal{O}}$ and consider the morphism $\mathbb{A}^1_{\mathcal{O}}\to \mathbb{A}^1_{\mathcal{O}}$ given by $x\mapsto a\cdot x$ with $\val (a)>0$, it does not preserve $p_{\mathcal{O}}$.
\end{remark}

\begin{lemma}
Let $(V_1,p_1),(V_2,p_2)\in \aGVar[K]$, $f:(V_1,p_1)\to (V_2,p_2)$ a morphism, and $\mathcal{V}_i=\Phi^{\aff}_K(V_i,p_i)$, for $i=1,2$. Then $\Phi^{\aff}_K(f)$ is residually dominant, i.e. $\Phi^{\aff}_K(f)_{k_K}:(\mathcal{V}_1)_{k_K}\to (\mathcal{V}_2)_{k_K}$  is dominant.
\end{lemma}
\begin{proof}
By definition the following commutes
\[\xymatrix
{
\mathcal{V}_1(\mathcal{O})\ar[r]\ar[d]^{r_1} & \mathcal{V}_2(\mathcal{O})\ar[d]^{r_2}\\
(\mathcal{V}_1)_k\ar[r]^{\Phi^{\aff}(f)_k} & (\mathcal{V}_2)_k
}
\]
Since, by Proposition \ref{P:prod-two-affines}(4), $(r_i)_*p_i$ is Zariski dense in $(\mathcal{V}_i)_k$, for $i=1,2$, $\Phi^{\aff}(f)_k$ is dominant. By faithfully flat descent (see \cite[Appendix C]{gortz}), $\Phi^{\aff}(f)_{k_K}$ is also dominant.
\end{proof}

This leads to the following definition.

\begin{definition}
Let $\aRdSch[K]$ be the category of affine schemes over $\mathcal{O}_K$ with residually dominant morphisms. I.e. the objects are affine schemes over $\mathcal{O}_K$ and morphisms are morphisms $f:\mathcal{V}\to\mathcal{W}$ over $\mathcal{O}_K$ with $f_{k_K}:\mathcal{V}_{k_K}\to \mathcal{W}_{k_K}$ dominant.  
\end{definition}

\begin{proposition}\label{C:fully-faithful}
After restricting the codomain of $\Phi^{\aff}_K$ to $\aRdSch[K]$, $\Phi^{\aff}_K$ is fully faithful.
\end{proposition}
\begin{proof}
Faithfulness is straight forward, indeed for every $(V_1,p_1),(V_2,p_2)\in\aGVar[K]$ and morphism $f:(V_1,p_1)\to (V_2,p_2)$, if we denote by $\Phi^{\aff}_K(f)$ the morphism $\Phi^{\aff}_K(V_1,p_1)\to \Phi_K(V_2,p_2)$ then by construction $(\Phi^{\aff}_K(f))_K=f$.

As for fullness, let $(V_1,p_1),(V_2,p_2)\in \aGVar[K]$ and $\Theta:\mathcal{V}_1\to \mathcal{V}_2$ a morphism, where $\mathcal{V}_i=\Phi^{\aff}(V_i,p_i)$ for $i=1,2$. 

Consider the following commutative diagram
\[
\xymatrix {
\mathcal{V}_1(\mathcal{O})\ar[r]^\Theta \ar[d]^r & \mathcal{V}_2(\mathcal{O}) \ar[d]^r \\
(\mathcal{V}_1)_k \ar[r]^{\Theta_k} & (\mathcal{V}_2)_k
}
\]

\begin{claim}
$\Theta_*p_1=p_2$.
\end{claim}
\begin{claimproof}
We will use Lemma \ref{L:same-type-iff-dense}. Let $f\in K[V_2]^{p_2}$ and let $c\models p_2|K$. By Proposition \ref{P:properties of Phi(V,p)}(3), $\val(f(c))=0$ and thus $f(c)\otimes_{\mathcal{O}_K} k\neq 0$. Since $\Theta_k$ is dominant, $f(\Theta(c))\otimes k\neq \emptyset$. Since $r$ is surjective (Theorem \ref{T:Surjective}), $D_{\mathcal{V}_1}(f\Theta)(\mathcal{O})\neq \emptyset$ and as a result $\val(f(\Theta(c)))=0$, as needed.
%
%
%Let $\mathcal{U}\subseteq \mathcal{V}_2$ be a basic open affine. Since $\Theta_k$ is dominant, $(\Theta_k^{-1})(\mathcal{U}_k)$ is non-empty, and since $r$ is surjective (Theorem \ref{T:Surjective}),  $r^{-1}((\Theta_k^{-1})(\mathcal{U}_k))$ is also non-empty. Since $r^{-1}((\Theta_k^{-1})(\mathcal{U}_k))$ is the set of  $\mathcal{O}$-points of an open subscheme, by Lemma \ref{L:same-type-iff-dense}, $p_1$ is concentrated on $r^{-1}((\Theta_k^{-1})(\mathcal{U}_k))$. Going the other direction in the diagram, $\Theta_*p_1$ is concentrated on $\mathcal{U}(\mathcal{O})$. By Lemma \ref{L:same-type-iff-dense}, $\Theta_*p_1=p_2$.
%
\end{claimproof}

We may now lift $\Theta$ to $\Theta_K: (V_1,p_1)\to (V_2,p_2)$ so $\Phi$ is full. 
\end{proof}

\subsection{Finiteness Conditions}\label{ss:finiteness conditions}
A natural question would be, given $(V,p)\in\aGVar[K]$ when is $\Phi^{\aff}(V,p)$ of finite type over $\mathcal{O}_K$? We are still not able to answer this question, but we are able to describe some related notions. We will first need to recall the notion of a strongly stably dominated type. It was first defined in \cite[Section 2.3]{Metastable} and later shown in \cite[Proposition 8.1.2]{Non-Arch-Tame} to be equivalent to the following when the type is concentrated on a variety.

Let $q$ be a definable type on a variety $V$ over a field. Write $\dim(q)$ for the dimension of the Zariski closure of $q$.
\begin{definition}
Let $q$ be an $A$-definable type on a variety $V$ over a valued field. Let $F$ be a valued field with $A\leq \dcl(F)$. Then $q$ is strongly stably dominated if $\dim(q)=\dim(g_*q)$ for some $F$-definable map $g$ into a variety over the residue field.
\end{definition}

For simplicity, assume that $V$ is affine. Consider the following finiteness conditions on $\Phi^{\aff}(V,p)$:
\begin{list}{A}{}
\item[(A)] $\Phi^{\aff}(V,p)$ is of finite type over $\mathcal{O}_K$;
\item[(B)] The polynomially convex hull of $p$
\[C(p):=\{x\in V: \val (f(x))\geq 0\text{ for all } f\in K[V]^p\}\]
is definable; 
\item[(C)] $p$ is strongly stably dominated.
\end{list}

If $\Phi^{\aff}(V,p)$ is of finite type over $\mathcal{O}_K$, then $\Phi^{\aff}(V,p)(\mathcal{O})$ is definable. Since $V$ is affine, $C(p)$ is definably isomorphic to $\Phi^{\aff}(V,p)(\mathcal{O})$. Hence $(A)$ implies $(B)$ and we will show that $(B)$ implies $(C)$.

\begin{definition}\cite[Definition 2.2.2]{Non-Arch-Tame}
Let $X=\varprojlim_i X_i$ be a pro-definable set, i.e. an inverse limit of definable sets. We say that $X$ is \emph{iso-definable} if for some $i_0$ the maps $X_i\to X_{i'}$ are bijections for all $i\geq i'\geq i_0$.
\end{definition}

\begin{fact}\label{F:iso-def}\cite[Corollary 2.2.4]{Non-Arch-Tame}
Let $X$ be a pro-definable set. Then $X$ is iso-definable if and only if $X$ is in (pro-definable) bijection with a definable set.
\end{fact}

\begin{proposition}\label{P:Iso-def-strongly-stab}
Let $(V,p)\in \aGVar[K]$ and $\mathcal{V}=\Phi^{\aff}(V,p)$. If $\mathcal{V}(\mathcal{O})$ is iso-definable then $p$ is strongly stably dominated and stably dominated via $r:\mathcal{V}(\mathcal{O})\to \mathcal{V}_k$.
\end{proposition}
\begin{proof}
Since $\mathcal{V}$ is an inverse limit of affine schemes of finite type over $\mathcal{O}_K$, we may write \[\mathcal{V}(\mathcal{O})=\varprojlim_i \mathcal{V}_i(\mathcal{O}).\] $\mathcal{V}(\mathcal{O})$ is iso-definable so we may assume that for every $i$, the map $\mathcal{V}(\mathcal{O})\to \mathcal{V}_i(\mathcal{O})$ is a bijection. 
Since $\mathcal{V}$ is integral, we may assume that so are the $\mathcal{V}_i$. Furthermore, since each of the $\mathcal{V}_i(\mathcal{O})$ are non-empty, the $\mathcal{V}_i$ are faithfully flat over $\mathcal{O}_K$ by Proposition \ref{P:faith-flat}. Consequently, by Fact \ref{F:dim-generic=special},
\[\dim (\mathcal{V}_i)_{K}=\dim (\mathcal{V}_i)_k,\]
for all $i$. For any $i$ consider the following commutative diagram

\[\xymatrix
{
\mathcal{V}(\mathcal{O})\ar[r]^r\ar[d]^{\pi_i} & \mathcal{V}_k\ar[d]^{(\pi_i)_k}\\
\mathcal{V}_i(\mathcal{O})\ar[r]^{r_i} & (\mathcal{V}_i)_k
}
\]
Since $\pi_i$ is bijective, and $r_i$ is surjective by Theorem \ref{T:Surjective}, $(\pi_i)_k$ is also surjective. Consequently, since $r_*p$ is Zariski dense in $\mathcal{V}_k$ by Proposition \ref{P:properties of Phi(V,p)}, $(r_i)_*p_i$ is Zariski dense in $(\mathcal{V}_i)_k$, where $p_i:=(\pi_i)_*p$, and $((\pi_i)_k)_*r_*p=(r_i)_*p_i$. As a result, $(\pi_i)_k$ is a dominant morphism. As a result, since 
\[\dim\mathcal{V}_K=\dim(\mathcal{V}_i)_K=\dim (\mathcal{V}_i)_k\leq \dim \mathcal{V}_k\leq \dim\mathcal{V}_K,\]
we conclude that $\dim\mathcal{V}_K=\dim\mathcal{V}_k$. So by Proposition \ref{P:actually stab.dom}, we conclude that $p$ is stably dominated by $r_*p$ via $r$.

Finally, since $\dim ((r_i)_*\pi_*p)=\dim(p)$, $p$ is strongly stably dominated by the definition of strongly stable domination given above.
\end{proof}

\begin{corollary}\label{C:conv-hull-to-strongly-stab}
Let $(V,p)\in \aGVar[K]$. If the polynomially convex hull of $p$, $C(p)$, is definable then $p$ is strongly stably dominated.
\end{corollary}
\begin{proof}
Let $\mathcal{V}=\Phi^{\aff}(V,p)$. Denote by $\phi:V\to \mathcal{V}_{K}$ the isomorphism and notice that $\phi(C(p))=\mathcal{V}(\mathcal{O})$. As before, we may write \[\mathcal{V}(\mathcal{O})=\varprojlim_i \mathcal{V}_i(\mathcal{O})\] and corresponding to it \[\phi=\varprojlim_i \phi_i.\]
Since $\phi$ is injective on $V$ we may assume that all the $\phi_i$ are injective. Thus $C(p)=\bigcap_i \phi_i^{-1}(\mathcal{V}_i(\mathcal{O}))$ but $C(p)$ is definable so $\phi(C(p))=\mathcal{V}_i(\mathcal{O})$ for some $i$, hence $\mathcal{V}(\mathcal{O})$ is iso-definable by Fact \ref{F:iso-def} and $p$ is strongly stably dominated by Proposition \ref{P:Iso-def-strongly-stab}.
\end{proof}

We tend to think that $(C)$ is probably equivalent to $(B)$, at least for curves, but we still do not have a proof.

\begin{example}
By Example \ref{E:p_O-minimal-val}, $p_{\mathcal{O}}^n$ induces the Gauss valuation, and thus $\Phi(\mathbb{A}^n,p_{\mathcal{O}}^n)=\mathbb{A}^n(\mathcal{O})$.
\end{example}

\begin{example}{\cite[Example 3.2.2]{Non-Arch-Tame}}
Let $K=\mathbb{C}(t)^{alg}$ with $\val (t)=1$ and \[H_n=\{(x,y)\in \mathbb{A}^2_K: \val (y-\sum_{i=0}^n (tx)^i)/i!)\geq n+1\}.\] There is a generically stable type $p$ concentrated on $H=\bigcap_n H_n$ which is Zariski dense in $\mathbb{A}^2_K$ which is not strongly stably dominated. As a result $\Phi(\mathbb{A}^2_K,p)$ is not of finite type over $\mathcal{O}_K$.
\end{example}

\begin{question}\label{Q:of finite type for strong}
Let $(V,p)\in \aGVar[K]$ with $p$ strongly stably dominated. Is $\Phi^{\aff}(V,p)$ of finite type over $\mathcal{O}_K$? Is $\Phi^{\aff}(V,p)(\mathcal{O})$ iso-definable?
\end{question}

\section{Generically Stable Groups}\label{s:gen stable grps}
We recall the definition of a generically stable group. Although everything may be defined in a broader context, we restrict ourselves to ACVF, see \cite{Metastable}.

Let $G$ be an ($\infty$-)definable group, $p$ an $A$-definable type on $G$ and $g\in G$. The \emph{left translate} of $p$ by $g$, $gp$,  is the definable type such that for any $A\cup {g}\subseteq A^\prime$, \[d\models p|A^\prime \Leftrightarrow gd\models gp|A^\prime.\] Similarly the right translate $pg$.

\begin{definition}
Let $G$ be an ($\infty$-)definable group. A definable type $p$ is \emph{left generic} if for any $A=\acl(A)$ over which it is defined and $g\in G$, $pg$ is definable over $A$. Similarly right generic.
\end{definition}

\begin{fact}\cite[Lemma 3.11]{Metastable}
Let $G$ be an $\infty$-definable group.
\begin{enumerate}
\item Right generics have boundedly many left translates. Similarly left generics.
\item Any generically stable left generic is also right generic. Any two generically stable generics differ by a left($\backslash$right) translation.
\item Assume $G$ admits a generically stable generic. Let $G^0$ be the intersection of all definable subgroups of finite index. It is of bounded index and called the \emph{connected component} of $G$.
\end{enumerate}
\end{fact}

\begin{fact}\cite[Remark 3.12]{Metastable}
Let $p$ be a generically stable generic on $G$. The following are equivalent:
\begin{enumerate}
\item $p$ is the unique generic type of $G$.
\item For all $g\in G$, $gp=p$.
\item $G=G^0$.
\end{enumerate}
\end{fact}

\begin{definition}\label{D:gen-stab-grp}
A ($\infty$)-definable group we will be called \emph{generically stable} if it has a generically stable generic.
\end{definition}

\begin{remark}
It is known that in NIP structures, a definable group satisfies Definition \ref{D:gen-stab-grp} if and only if it is generically stable in the sense of \cite[Definition 6.3]{nip-invariant}. In \cite[Definition 4.1]{Metastable}, Hrushovski and Rideau-Kikuchi call such $\infty$-definable groups stably dominated groups in the general metastable setting.
\end{remark}

We wish to understand generically stable $(\infty)$-definable groups. By the following Fact \ref{F:why-subgroup-ofalg}, at least for a start we may restrict ourselves to generically stable Zariski dense subgroups of algebraic groups.

\begin{definition}\cite[Corollary 2.37]{Metastable}
Let $C$ be a substructure of a model of ACVF. A $C$-definable set $D$ will be called \emph{boundedly imaginary} if there exists $\beta\in\Gamma(C)$ and a definable surjective map $g:\left(\mathcal{O}/\beta\mathcal{O}\right)^n\to D$, for some $n\in\mathbb{N}$.
\end{definition}

\begin{fact}\cite[Corollary 6.4]{Metastable}\label{F:why-subgroup-ofalg}
Let $H$ be a connected, generically stable ($\infty$-)definable group. Then there exist an algebraic group $G$ and a definable homomorphism $f: H\to G$, with boundedly imaginary kernel. If $H$ is defined over $C=\acl(C)$, then $f$ can be found over $C$. 
\end{fact}

It would be interesting to classify the boundedly imaginary groups, for now we show an application for local fields.

A \emph{non-archimedean local field} is a field that is complete with respect to a discrete valuation and whose residue field is finite. Specifically it must be one the following:
\begin{enumerate}
\item The characteristic zero case: finite extensions of the p-adic numbers ($\mathbb{Q}_p$).
\item The characteristic $p$ case: the field of formal Laurent series $\mathbb{F}_q((T))$, for $q$ a power of $p$.
\end{enumerate}

\begin{lemma}
Let $L$ be a local field and $D$ an $L$-definable set. For every $\beta\in \Gamma(L)$, $\left(\mathcal{O}/\beta\mathcal{O}\right) (L)$ is finite. Hence if $D$ if boundedly imaginary, $D(L)$ is finite.
\end{lemma} 

\begin{proof}
In the non-archimedean topology on $L$, $\mathcal{O}(L)$ is compact and $\beta\mathcal{O}(L)$ is an open subgroup of $\mathcal{O}(L)$. The quotient $\left(\mathcal{O}/\beta\mathcal{O}\right) (L)$ is compact and discrete, thus finite. The result follows.
\end{proof}

\begin{corollary}
Let $L$ be a  local field and $H$ a connected generically stable group definable over $L$. Then there exists a definable homomorphism \[f:H(L)\to G(L)\] with $G$ an algebraic group over $L$, with finite kernel.
\end{corollary}
\begin{proof}
Since $H$ is definable over $L$, the proof of Fact \ref{F:why-subgroup-ofalg} shows that $f$ and $G$ are definable over a finite extension of $L$, thus so is the kernel. A finite extension of a local field is still a local field so the kernel is finite.
\end{proof}

\subsection{Examples}
\subsubsection{$GL_n(\mathcal{O})$}
Recall Example \ref{E:GL_n-mmp}, let $p:=\varphi_*p^{n^2}_{\mathcal{O}}\in \mathrm{GL}_n(\mathcal{O})$, where $\varphi$ is the definable map 
\[\bar{x}\mapsto \left( \bar{x}, \frac{1}{\det (\bar{x})}\right),\]
 be the Zariski dense generically stable type that exhibits the mmp for $\mathrm{GL}_n$ from that example. It follows from \cite[Proposition 6.9]{Metastable} that $p$ is the unique generic of $\mathrm{GL}_n(\mathcal{O})$.

By the aid of the following we get automatically a wider class of examples.
\begin{fact}\cite[Corollary 4.5]{Metastable}\label{F:intersection with alg.grp is gen. stable}
Let $G$ be an algebraic group, $N$ an algebraic subgroup. Let $H$ be a definable subgroup of $G$ in ACVF, with $H$ generically stable. Then $H\cap N$ is generically stable.
\end{fact}

\begin{corollary}
Let $G\subseteq \mathrm{GL}_n$ be an affine subgroup scheme over $\mathcal{O}$ then $G(\mathcal{O})$ is generically stable.
\end{corollary}

\subsubsection{$\mathbb{G}_m$ and $\mathbb{G}_a$ are not Generically Stable}
On the other hand, the following are examples of definable groups which are not generically stable. 
\begin{fact}\label{F:Gm and Ga not generically stable}
$\mathbb{G}_m$ and $\mathbb{G}_a$ are not generically stable.
\end{fact}
\begin{proof}
One can easily see that any generically stable type concentrated on any of these groups must have unboundedly many translates.
\end{proof}

As a corollary we get the following observations.

\begin{proposition}\label{P:No_Ga_Gm=Abelian}
Let $G$ be an algebraic group over $K$, a model of ACVF, with a generically stable generic type. Then $G$ is an Abelian Variety.
\end{proposition}
\begin{proof}
By Chevalley's Theorem there exists a unique normal closed affine algebraic subgroup of $G$ such that the quotient is an Abelian variety. Thus it is sufficient to show that this affine subgroup is trivial.

Let $H$ be this subgroup. Since $G$ has a generically stable generic, it can not contain copies of $\mathbb{G}_a$ or $\mathbb{G}_m$, for otherwise by \Cref{F:intersection with alg.grp is gen. stable} they would be generically stable. Thus the same holds for $H$. The radical of $H$ is the maximal closed connected normal solvable subgroup of $H$, but every connected solvable affine group has an isomorphic copy of $\mathbb{G}_a$ or $\mathbb{G}_m$ \cite[Lemma 6.3.4]{lingrps}, so $H$ is semisimple. But every semisimple group is generated by closed isomorphic copies of $\mathbb{G}_a$ \cite[Theorem 8.1.5]{lingrps}. So $H$ is trivial.
\end{proof}

\begin{corollary}\label{C:abel-is-connected}
Let $G$ be an algebraic group over $K$, a model of ACVF, with a generically stable generic type. Then $G$ is a connected generically stable group, i.e. it has a unique generically stable generic.
\end{corollary}
\begin{proof}
Since $G$ is an Abelian variety, it is a divisible group. Since $G^0=Stab(p)$ is an intersection of finite-index subgroups and every finite divisible group is trivial, $G^0=G$ and $G$ has a unique generically stable generic.
\end{proof}

\subsubsection{Group Schemes of Finite Type over $\mathcal{O}_K$}

\begin{proposition}\label{P:grp scheme over O is gen stable}
Let $\mathcal{G}$ be an integral group scheme of finite type over $\mathcal{O}_K$. Then $\mathcal{G}(\mathcal{O})$ is generically stable.
\end{proposition}
\begin{proof}
Let $q$ be a generic type of $\mathcal{G}_k$ and let $p$ be the generically stable $K$-definable type on $\mathcal{G}(\mathcal{O})$ supplied by Proposition \ref{P:existence of gen stable generic}.
For $g\in \mathcal{G}(\mathcal{O})$, consider the definable type $gp$. On the face of it it is definable over $\acl(Kg)$, for genericity we need to show that it is $K$-definable.

Since $r:\mathcal{G}(\mathcal{O})\to\mathcal{G}_k$ is a group homomorphism, $r_*(gp)=r(g)r_*(p)=r(g)q$. Since $q$ is a generic type of $G_k$, $r(g)q$ is also a generic type of $\mathcal{G}_k$. By Lemma \ref{L:6.7fromMeta}, $gp$ is also $K$-definable.
\end{proof}

\subsection{The Maximum Modulus Principle for Group Schemes over $\mathcal{O}_K$}\label{ss:mmp_for_O_schemes}
Recall the maximum modulus principle (see Definition \ref{D:mmp}). Hrushovski and Rideau-Kikuchi give in \cite[Proposition 6.9]{Metastable} a characterization of connected generically stable subgroups of affine algebraic groups in terms of the maximum modulus principle. What follows is a generalization.

Let $p(x), q(y)$ be types, define $(p\times q)(x,y):=p(x)\cup q(y)$.
\begin{fact}\cite[Theorem 14.13]{StableDomination}\label{F:mmp from stabledomination} 
Let $P(x,y)$ be a polynomial over the algebraically closed valued field $K$, and $p,q$ be generically stable types in the field sort over $K$. Then $|P(x,y)|$ has a maximum $\gamma_{\max}\in \Gamma(K)$ on $p\times q$. Also, \[(a,b)\models p\otimes q\Rightarrow |P(a,b)|=\gamma_{\max}.\]
\end{fact}

\begin{theorem}\label{T:unique-generic-only-alg-group}
Let $K$ be a model of ACVF, $\mathcal{G}$ an integral separated group scheme of finite type over $\mathcal{O}_K$ and $p$ a $K$-definable type concentrated on $\mathcal{G}(\mathcal{O})$. The following are equivalent:
\begin{enumerate}
\item $\mathcal{G}$ has the mmp w.r.t. $p$;
\item $\mathcal{G}(\mathcal{O})$ is generically stable with $p$ as its unique generically stable generic type.
\end{enumerate}
\end{theorem}
\begin{proof}
$(1)\implies (2)$. Assume that $\mathcal{G}$ has the mmp w.r.t. $p$, in particular it is generically stable by Lemma \ref{L:implications-of-mmp}. We will prove that $gp=p$ for every $g\in \mathcal{G}(\mathcal{O})$.

Let $\mathcal{V}\subseteq \mathcal{G}$ be an affine open subscheme. For $g\in \mathcal{G}(\mathcal{O})$, denote by $g\cdot \mathcal{V}$ the image of $\mathcal{V}$ under the left translation map, it is an open affine subscheme over $\mathcal{O}$.  By Proposition \ref{P:mmp iff for affine cover}, $gp$ is concentrated on $\mathcal{V}(\mathcal{O})$ for every $g\in \mathcal{G}(\mathcal{O})$ (indeed, $p\in g^{-1}\cdot \mathcal{V}(\mathcal{O})$). Let $f$ be a regular function on $\mathcal{V}_\mathbb{K}$, and $b\models p|L$, where $L$ is a model over which $f$ is defined, if \[\val (f(b))=\gamma_f\] then by the maximum modulus principle \[\val (f(gb))\geq \gamma_f.\] On the other hand, since $p$ and $g^{-1}p$ are concentrated on $g^{-1}\cdot \mathcal{V}(\mathcal{O})$ and $f(gx)$ is a regular function on $(g^{-1}\cdot \mathcal{V})_\mathbb{K}$ then by the maximum modulus principle with respect to $p$ for $g^{-1}\mathcal{V}$, \[\val (f(b))=\val (f(gg^{-1}b))\geq \val (f(gb)),\]

so $p=gp$.

$(2)\implies (1)$. Let $\mathcal{G}=\bigcup_{i=1}^n \mathcal{V}_i$ be an affine open cover by affine open subschemes over $\mathcal{O}_K$.

\begin{claim*}
The type $p$ is concentrated on every $\mathcal{U}(\mathcal{O})$, where $\mathcal{U}$ is an open subscheme of $\mathcal{G}$ with $\mathcal{U}(\mathcal{O})\neq \emptyset$.
\end{claim*}
\begin{claimproof}
Let $\mathcal{U}\subseteq \mathcal{G}$ be an open subscheme of $\mathcal{G}$ with $\mathcal{U}(\mathcal{O})\neq\emptyset$. 
The proof mimics the proof of Proposition \ref{P:grp scheme over O is gen stable}, we show that the generic type found there can be found on $\mathcal{U}(\mathcal{O})$.

Since $p$ is translation invariant and $r:\mathcal{G}(\mathcal{O})\to \mathcal{G}_k$ is surjective by Theorem \ref{T:Surjective}, $r_*p$ is also translation invariant. Since $\mathcal{G}_k$ is an algebraic group, $r_*p$ is its unique generic type.

By Proposition \ref{P:existence of gen stable generic}, there is a generically stable type $q$ concentrated on $\mathcal{U}(\mathcal{O})$ such that $r_*q=r_*p$ (indeed, $\mathcal{G}_k$ is irreducible so $r_*p$ is also the generic type of the open subvariety $\mathcal{U}_k\subseteq \mathcal{G}_k$). 
Continuing as in the proof of Proposition \ref{P:grp scheme over O is gen stable}, $q$ is a generically stable generic type concentrated on $\mathcal{G}(\mathcal{O})$. Since $p$ is the unique generic type, $p=q$ and thus $p$ is concentrated on $\mathcal{U}(\mathcal{O})$. 
\end{claimproof}

We need to show that for each $\mathcal{V}_i$ such that $\mathcal{V}_i(\mathcal{O})\neq\emptyset$, $\mathcal{V}_i$ has the mmp w.r.t. p. Let $\mathcal{V}:=\mathcal{V}_i$ be such an affine open subscheme. By the Claim above, $p$ is concentrated on $\mathcal{V}(\mathcal{O})$.

Consider the multiplication map $m:\mathcal{G}\times \mathcal{G}\to \mathcal{G}$, let $h\in \mathcal{V}(\mathcal{O})$ and $c\models p|\mathbb{K}$. Since $h=hc^{-1}\cdot c$ and $m^{-1}(\mathcal{V})$ is covered by finitely many basic open affine subschemes, there exist $k,j$ (with out loss of generality $k=1,\, j=2$) and a basic open affine subscheme $D_{\mathcal{V}_1\times \mathcal{V}_2}(\alpha)\subseteq \mathcal{V}_1\times \mathcal{V}_2$ such that $m$ restricts to $D_{\mathcal{V}_1\times \mathcal{V}_2}(\alpha)\to \mathcal{V}$ and $(hc^{-1},c)\in D_{\mathcal{V}_1\times \mathcal{V}_2}(\alpha)(\mathcal{O})$, where \[D_{\mathcal{V}_1\times \mathcal{V}_2}(\alpha)(\mathcal{O})=\{(x,y)\in (\mathcal{V}_1\times \mathcal{V}_2)(\mathcal{O}):\val (\alpha(x,y))=0\}.\]

\begin{claim*}
$p\otimes p$ is concentrated on $D_{\mathcal{V}_1\times \mathcal{V}_2}(\alpha)(\mathcal{O})$.
\end{claim*}
\begin{claimproof}
By the previous claim, $p$ is concentrated on each $\mathcal{V}_j(\mathcal{O})$ with $\mathcal{V}_j(\mathcal{O})\neq\emptyset$, and thus $p\otimes p$ is concentrated on $\mathcal{V}_1(\mathcal{O})\times \mathcal{V}_2(\mathcal{O})$. Since $(hc^{-1},c)\models p\times p$ is concentrated on $D_{\mathcal{V}_1\times \mathcal{V}_2}(\alpha)(\mathcal{O})$, and so $\val (\alpha(hc^{-1},c))=0$, the result follows by Fact \ref{F:mmp from stabledomination}. 
\end{claimproof}

Let $f$ be a regular function on $\mathcal{V}_\mathbb{K}$. Hence \[f(h)=f(hc^{-1}\cdot c)=\frac{\sum f_i(hc^{-1})g_i(c)}{\alpha^k(hc^{-1},c)},\] where $\sum f_i(x)g_i(y)$ is a regular function on $(\mathcal{V}_1)_\mathbb{K}\times (\mathcal{V}_2)_\mathbb{K}$. 
Using Fact \ref{F:mmp from stabledomination} and the fact that $p\otimes p$ is concentrated on $D_{\mathcal{V}_1\times \mathcal{V}_2}(\alpha)(\mathcal{O})$,
\[\val (f(h))=\val (f(hc^{-1}\cdot c))=\val \left(\frac{\sum f_i(hc^{-1})g_i(c)}{\alpha^k(hc^{-1},c)}\right)\]
\[=\val \left(\sum f_i(hc^{-1})g_i(c)\right)\geq \val \left(\sum f_i(a)g_i(b)\right)\]
\[=\val \left(\frac{\sum f_i(a)g_i(b)}{\alpha^k(a,b)}\right)=\val (f(a\cdot b)),\]
for $(a,b)\models p\otimes p|\mathbb{K}$. Since $p$ is generic, $a\cdot b\models p|\mathbb{K}$ and thus \[\left(d_px)(\val (f(h))\geq \val (f(x))\right).\]

\end{proof}

In the above proof we used the fact that $\mathcal{G}$ is of finite type over $\mathcal{O}_K$ in order to get that $\dim \mathcal{G}_K =\dim \mathcal{G}_k$. The first direction, though, follows through exactly to prove the following:

\begin{proposition}\label{P:unique-generic-pro-def alg}
Let $\mathcal{G}$ be a quasi-compact separated irreducible group scheme over $\mathcal{O}_K$ and $p$ a $K$-definable type concentrated on $\mathcal{G}(\mathcal{O})$. If $\mathcal{G}$ has the mmp w.r.t. $p$ then $\mathcal{G}(\mathcal{O})$ is generically stable with $p$ as its unique generically stable generic type.
\end{proposition}

\begin{remark}
Hrushovski and Rideau-Kikuchi, in \cite[Theorem 6.11]{Metastable}, use the functor $\Phi^{\aff}$ to associate with every affine algebraic group  and a generically stable generic of a type-definable subgroup and an affine group scheme over $\mathcal{O}_K$. Since we do not yet know if $\Phi^{\aff}$ can be extended to a functor on general varieties we could not follow a similar path. We aim to continue to follow this path in a subsequent paper.
\end{remark}

\subsection{Universally Closed}
Although not directly connected to the rest of the paper (though there could be some future connection to algebraic groups with a unique generically stable generic type), we think this result might be interesting in its own right.
%
%Two questions arise:
%\begin{enumerate}
%\item When is an Abelian variety generically stable?
%\item If $G$ is a generically stable connected algebraic group does it have a "model"? (in the sense of a proper group scheme over $\mathcal{O}$ with generic fibre isomorphic to $G$).
%\end{enumerate}
%
%We will answer both of these questions.
%
%Abelian varieties need not be generically stable. As an example, consider an elliptic curve with bad reduction. Specifically, the curve given by $y^2=x^3+5$ over $\mathbb{Q}_5^{alg}$. It is an algebraic group, call it $E$. If it were generically stable, by \Cref{T:Main Theorem} there would exist a group scheme $\mathcal{H}$ with $\mathcal{H}_K=E$, but then its reduction $H_k$ is a group, contradicting the bad reduction of $E$.
%
%If the Abelian variety comes as the $K$-points $\mathcal{H}_K(K)$ of some Abelian scheme $\mathcal{H}$ over $\mathcal{O}_K$ then by \Cref{T:grp scheme over O is gen stable} $\mathcal{H}(\mathcal{O})=\mathcal{H}_K$ is generically stable. In fact since $\mathcal{H}_K$ is a divisible group it is also a connected generically stable group, because working in ACF, if $p$ is a generic type of $\mathcal{H}_K$ then $\mathcal{H}_K/Stab(p)$ is a finite divisible group and hence trivial.

We recall the following formulation for the valuation criterion for universally closed morphisms and deduce consequences.

\begin{fact}\label{F:val_crit}\cite[Tag 0894]{stacks-project}
Let $f : X \to S$ and $h : U \to X$ be morphisms of schemes.
Assume that $f$ and $h$ are quasi-compact and that $h(U)$ is dense in $X$.
If given any commutative diagram
\[
\xymatrix {
\Spec L \ar[r] \ar[d] & U \ar[r] & X \ar[d] \\
\Spec R \ar[rr] \ar@{-->}[rru] & & S
}
\]
where $R$ is a valuation ring with field of fractions $L$, there
exists a unique dotted arrow making the diagram commute, then $f$
is universally closed.
\end{fact}

\begin{corollary}\label{C:val_crit}
Let $F$ be a valued field and $\mathcal{V}$ a quasi-compact separated irreducible scheme over $\mathcal{O}_F$ with an $\mathcal{O}_F$-point. Then $\mathcal{V}$ is universally closed over $\mathcal{O}_F$ if and only if  $\mathcal{V}_R(R)=\mathcal{V}_L(L)$ for every field extension $F\subseteq L$ and $\mathcal{O}_F\subseteq R$ a valuation ring on $L$.
\end{corollary}
\begin{proof}
Assume that $\mathcal{V}$ is universally closed over $\mathcal{O}_F$. Since universally closed is stable under base change \cite[Remark 14.50]{gortz}, $\mathcal{V}_R$ is universally closed over $R$. By the valuative criterion for universally closed morphisms, for every $L$-point, $\Spec L\to \mathcal{V}_R$, there exists a unique dotted arrow making the following commute
\[
\xymatrix {
\Spec L \ar[r] \ar[d] &\mathcal{V}_R \ar[d] \\
\Spec R \ar[r] \ar@{-->}[ru] &  \Spec R
}
\]
and thus  $\mathcal{V}_R(R)=\mathcal{V}_L(L)$.

For the other direction, since $\mathcal{V}$ is irreducible, by \Cref{F:val_crit} it is enough to consider commutative diagrams of the sort
\[\xymatrix
{
\Spec L\ar[r]\ar[d]
& \mathcal{V}\ar[d]\\
\Spec R\ar[r]
& \Spec \mathcal{O}_F
}\]
where $\Spec L$ gets sent to the generic point of $\mathcal{V}$ and $L$ is the fraction field of $R$. The diagram commutes so there is a point in $\Spec R$ which gets sent to the generic point of $\Spec \mathcal{O}_F$, but this implies that the generic point of $\Spec R$ also gets sent to the generic point of $\Spec \mathcal{O}_F$.  As $\Spec \mathcal{O}_F$ is reduced, $\mathcal{O}_F\subseteq R$. So the result follows from Fact \ref{F:val_crit} and the assumption. 
\end{proof}

\begin{lemma}\label{L:uni-sep-iff-all-limits}
Let $F$ be a valued field and $\mathcal{V}=\varprojlim_i \mathcal{V}_i$ be a quasi-compact separated irreducible scheme over $\mathcal{O}_F$ with an $\mathcal{O}_F$-point, and assume that the $\mathcal{V}_i$ are separated over $\mathcal{O}_F$. If $\mathcal{V}_F$ is of finite type over $F$ then $\mathcal{V}$ is universally closed over $\mathcal{O}_F$ if and only if there exists $i_0$ such that for all $i\geq i_0$ $\mathcal{V}_i$ is universally closed over $\mathcal{O}_F$.
\end{lemma}
\begin{proof}
Since $\mathcal{V}_F$ is of finite type over $F$ we may assume that all the transition maps in $\mathcal{V}_F=\varprojlim_i (\mathcal{V}_i)_F$ are isomorphisms. We use Corollary \ref{C:val_crit}. Let $F\subseteq L$ be an extension of fields and $\mathcal{O}_F\subseteq R$ a valuation ring on $L$, and consider the following commutative diagram
\[\xymatrix
{
\mathcal{V}_R(R)\ar[d]^{\pi_i} \ar[r]^\iota & \mathcal{V}_L(L)\ar[d]^{(\pi_i)_L}\\
(\mathcal{V}_R)_i(R)\ar[r]^{\iota_i}& (\mathcal{V}_i)_L(L)
}\]
Since $(\pi_i)_L$ is a bijection, if $\iota$ is a bijection so is $\iota_i$. If $\iota_i$ is a bijection for all $i$ then $\mathcal{V}_i(\mathcal{O}_L)\to \mathcal{V}_j(\mathcal{O}_L)$ is a bijection for all $i\geq j$ so $\pi_i$ is also a bijection. Hence $\iota$ is a bijection.
\end{proof}

\begin{proposition}\label{P:uni-closed}
Let $F$ be a non-trivially valued field and $\mathcal{V}$ a quasi-compact separated irreducible scheme over $\mathcal{O}_F$ with an $\mathcal{O}_F$-point such that $\mathcal{V}_F$ is geometrically irreducible. If $\mathcal{V}_F$ is proper over $F$ and $\mathcal{V}(\mathcal{O})= \mathcal{V}_F$ then $\mathcal{V}$ is universally closed over $\mathcal{O}_F$.
\end{proposition}
\begin{proof}
By Lemma \ref{L:uni-sep-iff-all-limits} we may assume that $\mathcal{V}$ is of finite type over $\mathcal{O}_F$. Since the assumptions are stable under base-change (see \cite[Appendix C]{gortz}), by faithfully flat descent (see \cite[Remark 14.50]{gortz}), we may assume that $F$ is algebraically closed (i.e. since $F$ non-trivially valued, a model of ACVF). Consequently, by the assumptions $\mathcal{V}(\mathcal{O}_F)=\mathcal{V}_F(F)$.

Let $F\subseteq L$ be a field extension and $R$ a valuation ring of $L$ containing $\mathcal{O}_F$, as in Corollary \ref{C:val_crit}. Note that we may assume that $L$ is algebraically closed, and hence, since $F$ is not trivially valued, a model of ACVF. Indeed, let $L^{alg}$ be the algebraic closure of $L$ and $R'$ an extension of $R$ to $L^{alg}$. Since the sentence $\mathcal{V}_{R'}(R')=\mathcal{V}_{L^{alg}}(L^{alg})$ is a universal sentence, it holds in $L$ as well so  so $\mathcal{V}(R)=\mathcal{V}_L(L)$.

If $R\cap F=\mathcal{O}_F$ then $(F,\mathcal{O}_F)$ is a substructure of $(L,R)$ and hence by model completeness $\mathcal{V}_R(R)=\mathcal{V}_L(L)$.

If $\mathcal{O}_F\subsetneq R\cap F\subsetneq F$ then since $\mathcal{V}(\mathcal{O}_F)\to \mathcal{V}_F(F)$ factors through $\mathcal{V}_{R\cap F}(R\cap F)\to \mathcal{V}_F(F)$, the latter must be surjective as well (note that $R\cap F$ is valuation ring of $F$). As before we are now in the situation that $(F,R\cap F)$ is a substructure of $(L,R)$ hence by model completeness $\mathcal{V}_R(R)=\mathcal{V}_L(L)$.

If $R\cap F=F$ (i.e. $F\subseteq R$) then consider the following commutative diagram:

\[\xymatrix
{
\Spec L \ar@/^1pc/[rr]\ar[d] & \Spec \mathcal{V}_{F}\ar[r]\ar[d]& \Spec \mathcal{V}\ar[d]\\
\Spec R \ar[r] & \Spec F \ar[r] & \Spec\mathcal{O}_F
}
\]
By the universal property of fiber-product the above completes to
\[\xymatrix
{
\Spec L \ar[r]\ar[d] & \Spec \mathcal{V}_{F}\ar[r]\ar[d]& \Spec \mathcal{V}\ar[d]\\
\Spec R \ar[r]\ar@{-->}[ur]\ar@{-->}[urr] & \Spec F \ar[r] & \Spec\mathcal{O}_F
}
\]

In the above, the less horizontal dotted arrow exists since $\mathcal{V}_F$ is proper, and hence universally closed over $\Spec F$. The more horizontal arrow now exists easily. Hence $\mathcal{V}_R(R)=\mathcal{V}_L(L)$.

Finally, $\mathcal{V}$ is universally closed over $\mathcal{O}_F$ by Corollary \ref{C:val_crit}.
\end{proof}

\paragraph*{Acknowledgement}
I would like to thank my PhD advisor Ehud Hrushovski for his careful reading of previous drafts, his many ideas, our discussions and his support. Also for helping me define the correct notions and definitions in this paper. I would also like to thank Yuval Dor for reading some previous drafts, our discussions and many useful comments. Lastly, I would like to thank the anonymous referee whose comments were extremely helpful in improving this paper.

\bibliographystyle{plain}
\bibliography{metastable-groups}

\begin{thebibliography}{10}

\bibitem{gortz}
Ulrich G\"ortz and Torsten Wedhorn.
\newblock {\em Algebraic geometry {I}}.
\newblock Advanced Lectures in Mathematics. Vieweg + Teubner, Wiesbaden, 2010.
\newblock Schemes with examples and exercises.

\bibitem{Eli_Imag}
Deirdre Haskell, Ehud Hrushovski, and Dugald Macpherson.
\newblock Definable sets in algebraically closed valued fields: elimination of
  imaginaries.
\newblock {\em J. Reine Angew. Math.}, 597:175--236, 2006.

\bibitem{StableDomination}
Deirdre Haskell, Ehud Hrushovski, and Dugald Macpherson.
\newblock {\em Stable domination and independence in algebraically closed
  valued fields}, volume~30 of {\em Lecture Notes in Logic}.
\newblock Association for Symbolic Logic, Chicago, IL; Cambridge University
  Press, Cambridge, 2008.

\bibitem{Non-Arch-Tame}
Ehud Hrushovski and Fran{\c{c}}ois Loeser.
\newblock {\em Non-archimedean tame topology and stably dominated types},
  volume 192 of {\em Annals of Mathematics Studies}.
\newblock Princeton University Press, Princeton, NJ, 2016.

\bibitem{nip-invariant}
Ehud Hrushovski and Anand Pillay.
\newblock On {NIP} and invariant measures.
\newblock {\em J. Eur. Math. Soc. (JEMS)}, 13(4):1005--1061, 2011.

\bibitem{Metastable}
Ehud Hrushovski and Silvain Rideau-Kikuchi.
\newblock Valued fields, metastable groups.
\newblock {\em Selecta Math. (N.S.)}, 25(3):Paper No. 47, 58, 2019.

\bibitem{kuhlmann2}
Franz-Viktor Kuhlmann.
\newblock Elimination of ramification {I}: the generalized stability theorem.
\newblock {\em Trans. Amer. Math. Soc.}, 362(11):5697--5727, 2010.

\bibitem{kuhlmann}
Franz-Viktor Kuhlmann.
\newblock {\em Book on Valuation Theory}, 2016 (accessed September, 2016).
\newblock \url{http://math.usask.ca/~fvk/Fvkbook.htm}.

\bibitem{Marker}
David Marker.
\newblock {\em Model theory}, volume 217 of {\em Graduate Texts in
  Mathematics}.
\newblock Springer-Verlag, New York, 2002.
\newblock An introduction.

\bibitem{matsumura}
Hideyuki Matsumura.
\newblock {\em Commutative ring theory}, volume~8 of {\em Cambridge Studies in
  Advanced Mathematics}.
\newblock Cambridge University Press, Cambridge, second edition, 1989.
\newblock Translated from the Japanese by M. Reid.

\bibitem{nagata1966}
Masayoshi Nagata.
\newblock Finitely generated rings over a valuation ring.
\newblock {\em J. Math. Kyoto Univ.}, 5:163--169, 1966.

\bibitem{Robinson}
Abraham Robinson.
\newblock {\em Complete theories}.
\newblock North-Holland Publishing Co., Amsterdam-New York-Oxford, second
  edition, 1977.
\newblock With a preface by H. J. Keisler, Studies in Logic and the Foundations
  of Mathematics.

\bibitem{guidetonip}
Pierre Simon.
\newblock {\em A guide to {NIP} theories}, volume~44 of {\em Lecture Notes in
  Logic}.
\newblock Association for Symbolic Logic, Chicago, IL; Cambridge Scientific
  Publishers, Cambridge, 2015.

\bibitem{lingrps}
Tonny~Albert Springer.
\newblock {\em Linear algebraic groups}.
\newblock Springer Science \& Business Media, 2010.

\bibitem{stacks-project}
The {Stacks Project Authors}.
\newblock \itshape stacks project.
\newblock \url{http://stacks.math.columbia.edu}, 2016.

\end{thebibliography}

\end{document}